\documentclass[large]{interact}
\usepackage{geometry}
\usepackage{amsmath}
\usepackage{epstopdf}
\usepackage[caption=false]{subfig}
\usepackage{soul}
\usepackage{setspace}
\usepackage{tikz-cd}
\usepackage{tikz}
\usepackage{hyperref}
\usepackage{biblatex} 
\usepackage{comment}
\renewcommand\bibfont{\fontsize{10}{12}\selectfont}
\makeatletter
\def\NAT@def@citea{\def\@citea{\NAT@separator}}
\makeatother

\theoremstyle{plain}
\newtheorem{theorem}{Theorem}[section]
\newtheorem{lemma}[theorem]{Lemma}
\newtheorem{corollary}[theorem]{Corollary}

\theoremstyle{definition}

\theoremstyle{remark}
\newtheorem{remark}{Remark}

\begin{document}
\title{\vspace{-0.5 cm}Solutions to the Seiberg--Witten equations in all dimensions}
\author{
\name{Partha Ghosh}
}
\newcommand{\Addresses}{{
  \bigskip
  \footnotesize

  \textsc{Institut de Mathématiques de Jussieu-Paris Rive Gauche, Sorbonne Université, Paris, France}\par\nopagebreak
  \textit{E-mail address}: \texttt{Partha.Ghosh@imj-prg.fr}
}}
\maketitle
\begin{abstract}
This article explores solutions to a generalised form of the Seiberg--Witten equations in higher dimensions, first introduced by Fine and the author \cite{JP}. Starting with an oriented $n$ dimensional Riemannian manifold with a spin$^\mathbb{C}$-structure, we described an elliptic system of equations that recovers the traditional Seiberg-Witten equations in dimensions $3$ and $4$. The paper focuses on constructing explicit solutions of these equations in dimensions $5, 6$ and $8$, where harmonic perturbation terms are sometimes required to ensure solutions. In dimensions $6$ and $8$ we construct solutions on Kähler manifolds and relate these solutions to vortices. In dimension $5$, we construct solutions on the product of a closed Riemann surface and $\mathbb{R}^3$. The solutions are invariant in the $\mathbb{R}^3$ directions and can be related to vortices on the Riemann surface. A key issue in higher dimensions is the potential noncompactness of the space of solutions, in contrast to the compact moduli spaces in lower dimensions. In our solutions, this noncompactness is linked to the presence of certain odd-dimensional harmonic forms, with an explicit example provided in dimension $6$. 
\end{abstract}
\tableofcontents

\section{Introduction}\label{introdution}
\subsection{Background} A generalisation of the Seiberg--Witten equations in all dimensions were introduced by Fine and myself \cite{JP}. Starting with an $n$-dimensional oriented Riemannian manifold with a spin$^\mathbb{C}$-structure, we described an elliptic system of equations which recover the Seiberg--Witten equations when $n=3,4$. \par
The original Seiberg--Witten equations introduced in \cite{SW},
led quickly to a revolution in $3$- and $4$-dimensional differential geometry and they remain at
the forefront of research today. Shortly after their appearance, Witten showed how one could count solutions to the equations, defining an invariant of the underlying smooth $4$-manifold \cite{EW}. In higher dimensions, there is no need for a gauge theoretic approach to study smooth structures since, for example, the $h$-cobordism theorem holds \cite{Smale}. Instead, one might speculate that higher dimensional Seiberg--Witten equations could prove useful when studying manifolds with geometric structures. The fantasy is inspired by Taubes' \textit{tour de force} \cite{Taubes2,Taubes3,Taubes4} connecting Seiberg--Witten and Gromov--Witten invariants of symplectic four-manifolds. Taubes also proves that for a compact symplectic 4-manifold with $b_+>1$ the Seiberg--Witten invariant for the canonical spin$^\mathbb{C}$-structure is always $1$ \cite{Taubes1}. This gives an obstruction to the existence of symplectic structures. There is no known obstruction in higher dimensions, beyond the most obvious ones that there must be a degree $2$ cohomology class with non-zero top power and the manifold must admit an almost complex structure. One can hope that the higher dimensional Seiberg--Witten equations could tell us something about higher dimensional symplectic manifolds. We make more comments on it in \S\ref{Tabues' limit}.\par
This article is dedicated to constructing several solutions of these new higher-dimensional Seiberg--Witten equations. We construct solutions of the Seiberg--Witten equations in dimensions $5,6$ and $8.$ Considering our interest in the potential applications of these equations in symplectic geometry, we examined them on Kähler manifolds and discovered solutions. In some cases, we needed to add a suitable harmonic ``perturbation" term to the curvature equation for the construction of a solution. \par
One of the key features in both the $3$- and $4$-dimensional case is that the moduli space of the Seiberg--Witten equations over a closed manifold is compact. This compactness plays a crucial role in constructing the Seiberg–-Witten invariant. For $n>4$ although some initial bounds are established in \cite{JP}, these results alone are insufficient to guarantee that the solution space is compact. Instead, they leave open the possibility of ``bubbling" occurring. For our ansatz, we discovered that while the solution space can indeed be non-compact, this happens in a very specific way, depending on the existence of certain odd-dimensional harmonic forms. We provide a concrete example of this phenomenon in dimension $6$ in \S\ref{Compactness}. Nonetheless, one may hope that, modulo these exceptions, the solution space remains compact. 
\subsection{Overview of the main results}
We begin by fixing notation. Let $(M,g)$ be a Riemannian manifold of dimension $n$, which admits a spin$^\mathbb{C}$-structure.  Write $S \xrightarrow{} M^n$ for the spin bundle of the spin$^\mathbb{C}$-structure. When $n$ is even, $S =S_+ \oplus S_-$ splits into sub-bundles of positive and negative spinors. We write $c: \Lambda^* \xrightarrow{} \text{End}(S)$ for the Clifford action of differential forms on spinors. We follow the conventions of \cite{Friedrich}. In particular, (real) 1-forms act as skew-Hermitian endomorphisms. Meanwhile, in dimension $n=2m$, the volume form satisfies $i^mc(\text{dvol})) = \pm 1$ on $S_{\pm}$ whilst in dimension $n=2m-1$, $i^m c(\text{dvol}) = 1$ on all of $S$. 

Let $L(S)$ denote the line bundle associated to the Spin$^\mathbb{C}$-structure and $\mathcal{A}$ denote the set of unitary connections in $L(S)$ .Given $A \in \mathcal{A}$, we write $D_A$ for the associated Dirac operator. 

When $n$ is even, a spinor $\phi \in S_\pm$ defines a trace-free Hermitian endomorphism $E_\phi \colon S_\pm \to S_\pm$ via
\begin{equation}
E_\phi(\psi)= \left\langle \psi,\phi \right\rangle\phi - \frac{1}{r} |\phi|^2 \psi.
\label{Ephi}
\end{equation}
where $r$ is the rank of $S_\pm$. When $n$ is odd and $\phi \in S$ we use $E_\phi$ to denote the analogous trace-free endomorphism of $S$, where now $r$ is the rank of $S$.\par
The original Seiberg--Witten equations exhibit notable behavior when applied to Kähler surfaces, specifically in relation to vortex equations. For a hermitian line bundle $L$ (say $h$ is the hermitian metric) on a closed K\"ahler manifold $(X^n,\omega)$ of complex dimension $n$ and for any $\tau\in\mathbb{R},$ the \textit{$\tau$-vortex equations} for a unitary connection $A$ on $L$ and a section $\phi\in\Omega^0(X,L)$ are
\begin{align}
D_A^{0,1}\phi=0\label{Vortex 1}\\
  i\Lambda F_A+\frac{1}{2}\phi\otimes\phi^{*h}=\frac{1}{2}\tau \hspace{0.5 ex}\text{Id}\label{Vortex 2}\\
  F_A^{0,2}=0\label{Vortex 3}
\end{align}
$\Lambda$ is the $L^2$ adjoint of $(\wedge\omega).$\par 
Notice, equation \eqref{Vortex 3} says that $A$ defines a holomorphic structure on $L$ and then equation \eqref{Vortex 1} says that $\phi$ is a holomorphic section with respect to the holomoprhic structure induced by $A$. An equivalent way of thinking about the vortex equations is that we start with a holomorphic line bundle $(L, \bar\partial)$ and a holomorphic section $\phi$ and then look for a hermitian metric $h$ whose Chern connection (say $A$) solves equation \eqref{Vortex 2}. Any holomorphic connection (i.e., $F_A^{0,2}=0$) can be seen as a unitary connection for a hermitian metric, and a hermitian connection induces a unique unitary holomorphic connection (the Chern connection). Hence these two formalisms are not the same but equivalent. The moduli space of vortices can be seen as a fibration over the Picard variety \( \mathrm{Pic}_{L}(X) \):
\[
\mathcal{M}_{\text{vortex}}(L) \to \mathrm{Pic}_{L}(X),
\]
where the base space \( \mathrm{Pic}_{L}(X) \) parameterizes the holomorphic structures on \( L \) and the fiber over a point \( [L] \in \mathrm{Pic}_{L}(X) \) is the projective space \( \mathbb{CP}\big(H^0(X, L)\big) \), which is the space of holomorphic sections of \( L \), modulo the \( \mathbb{C}^* \)-action.\par 
Geometrically, each vortex pair $(A,\phi)$ determines an effective divisor of $X$, namely $\phi^{-1}(0).$ Conversely, an effective divisor determines a holomorphic line bundle together with a holomorphic section. If $\mathcal{D}$ is the effective divisor, $L_\mathcal{D}$ is the line bundle, and $\phi_D$ is the holomorphic section, then $\phi_\mathcal{D}^{-1}(0)$ is precisely $\mathcal{D}$. Furthermore, $\mathcal{D}$ determines an element of $H_{2n-2}(X,\mathbb{C})$ and thus by Poincar\'e duality an element $\eta_\mathcal{D}\in H^2(X,\mathbb{C})$. The first Chern class $c_1({L}_\mathcal{D})$ is represented by $\eta_\mathcal{D}.$ It follows that
\begin{align*}
    \text{deg}(L_\mathcal{D})=\int_X c_1(L_\mathcal{D})\wedge \omega^{n-1}=\#(\mathcal{D},[\omega^{n-1}])
\end{align*}
where $[\omega^{n-1}]$ is the $2$-cycle dual to $\omega^{n-1}$ and $\#(,)$ denotes the intersection pairing. Define Div$_+(X, \text{deg}(L))$ to be the set of effective divisors on $(X,\omega)$ such that their intersection pairing with $[\omega^{n-1}]$ is deg$(L)$, modulo linear equivalence. For $\tau>\frac{4\pi\text{ deg}(L)}{\text{Vol}(X)}$, $(A,\phi)\mapsto \phi^{-1}(0)$ defines a bijection on equivalence classes between $
\mathcal{M}_{\text{vortex}}(L)$ and Div$_+(X, \text{deg}(L))$ \cite{Brad}.\par
On closed Kähler surfaces, the Seiberg--Witten equations can be reduced, under certain conditions, to the vortex equations. This reduction was shown by Witten \cite{EW} as part of a broader connection between Seiberg--Witten theory and symplectic geometry. For example, in symplectic $4$-manifolds, Taubes established a correspondence where Seiberg--Witten solutions correspond to pseudoholomorphic curves \cite{Taubes4}, which are governed by the vortex-like equations under certain limits. This insight has allowed the Seiberg--Witten equations to be applied in tackling problems such as the Thom conjecture \cite{KM}, linking the solutions to invariants related to holomorphic curves and vortex configurations.\par
This motivates us to look for solutions to the Seiberg--Witten equations on higher dimensional K\"ahler manifolds and if possible find their correspondence with vortices. We studied the equations on K\"ahler three- and four-folds and indeed found solutions which can be described in terms of vortices on those K\"ahler manifolds.

In \S\ref{6d} we construct solutions of the $6$-dimensional Seiberg--Witten equations on a K\"ahler $3$-fold $(X,\omega).$ A fixed Riemannian metric on $X$ is implicit throughout this discussion. We choose two spinor bundles $S,\tilde{S}\rightarrow X$ coming from two (potentially different) spin$^\mathbb{C}$-structures on $X$. In dimension $6,$ Clifford multiplication gives the following isomorphisms:
\begin{align}
    c:i\Lambda^2\rightarrow i\mathfrak{su}(S_+)\label{Clifford 6 +}\\
    c:\Lambda^4\rightarrow i\mathfrak{su}(\tilde{S}_-)\label{Clifford 6 -}
\end{align}
Given $\phi\in\Gamma(S_+),\psi\in\Gamma(\tilde{S}_-),$ we write $q(\phi)$ and $q(\psi)$ for the corresponding differential forms $E_\phi$ and $E_\psi$ under \eqref{Clifford 6 +} and \eqref{Clifford 6 -}.
Let $\mathcal{A}$ and $\mathcal{B}$ denote the set of unitary connections on $L(S)$ and $L({\tilde{S}})$ respectively. The $6$-dimensional Seiberg--Witten equations for $\phi\in\Gamma(S_+),\psi\in\Gamma(\tilde{S}_-),\beta\in\Omega^3,A\in\mathcal{A}$ and $B\in\mathcal{B}$ are
\begin{align}
    \big(D_A+c(*\beta)\big)\phi=0\label{Dirac 6d +}\\
    F_A+2id^*\beta=q(\phi)\label{Curvature 6d +}\\
    \big(D_B+c(\beta)\big)\psi=0\label{Dirac 6d -}\\
    -{i}*F_B+2d\beta=q(\psi)\label{Curvature 6d -}
\end{align}
Since there are two spinors and two connections in play, the appropriate gauge group is now $\mathcal{G}=$ Map$(X,S^1\times S^1)$ where the first factor acts by pull back on $(A,\phi)$, the second by pull back on $(B,\psi)$ and both factors act trivially on $\beta.$ The above equations are elliptic modulo gauge \cite{JP}.\par
The canonical spinor bundles on $X$, say $S^{\text{can}}=S_+^{\text{can}}\oplus S_-^{\text{can}} $ are given by 
\begin{align*}
    S_+^{\text{can}}=\Lambda^0\oplus\Lambda^{0,2}\\
    S_-^{\text{can}}=\Lambda^{0,1}\oplus\Lambda^{0,3}
\end{align*}
The $(0,q)$-forms are induced from the unique almost complex structure $J$ compatible with the Riemannian and the K\"ahler metric on $X.$
Any other spin bundle on $X$ is obtained by twisting $S^{\text{can}}$ by a complex line bundle. Take two hermitian holomorphic line bundles $\mathcal{L}_0,\mathcal{L}_1\rightarrow X.$ We have
\begin{align*}
    S_+=\Lambda^0(X,{\mathcal{L}_0})\oplus\Lambda^{0,2}(X,\mathcal{L}_0)\\
    S_-=\Lambda^{0,1}(X,{\mathcal{L}_0})\oplus\Lambda^{0,3}(X,\mathcal{L}_0)
\end{align*}
and 
\begin{align*}
    \tilde{S}_+=\Lambda^0(X,{\mathcal{L}_1})\oplus\Lambda^{0,2}(X,\mathcal{L}_1)\\
    \tilde{S}_-=\Lambda^{0,1}(X,{\mathcal{L}_1})\oplus\Lambda^{0,3}(X,\mathcal{L}_1)
\end{align*}
For a closed K\"ahler $3$-fold $(X,\omega)$, we prove the following in \S\ref{Proof 6d}. 
\begin{theorem}\label{Theorem 6d}
    For a fixed pair of holomorphic structures on $\mathcal{L}_0$ and $\mathcal{L}_1$, there exists a non-trivial solution to the Seiberg--Witten equations \eqref{Dirac 6d +},\eqref{Curvature 6d +},\eqref{Dirac 6d -} and \eqref{Curvature 6d -} under the following conditions:
    \begin{align*}
    &1.\hspace{1 ex} \text{dim }H^0(X,\mathcal{L}_0)>0\\
    &2. \hspace{1 ex}c_1(K_X^{-1}\otimes \mathcal{L}_0^2)=a_0[\omega]\hspace{1 ex} \text{with $a_0<0$}\hspace{1 ex}\text{[this implies deg $(K_X^{-1}\otimes \mathcal{L}_0^2)<0$]}\\
    &3. \hspace{1 ex} \text{dim }H^0(X,K_X\otimes\mathcal{L}_1^{-1})>0\\
    &4. \hspace{1 ex}c_1(K_X^{-1}\otimes \mathcal{L}_1^2)=a_1[\omega]\hspace{1 ex} \text{with $a_1>0$}\hspace{1 ex}\text{[this implies deg $(K_X^{-1}\otimes \mathcal{L}_1^2)>0$]}
\end{align*}
\end{theorem}
For the non-trivial solution mentioned in the theorem, $\phi\in\Omega^0(X,\mathcal L_0)$ and $\psi\in\Omega^{0,3}(X,\mathcal L_1)\cong \Omega^0(X,K_X^{-1}\otimes\mathcal{L}_1).$ Moreover we can identify the space of the constructed solutions modulo gauge with $\mathcal{M}_{\text{vortex}}(\mathcal{L}_0)\times \mathcal{M}_{\text{vortex}}(K_X\otimes\mathcal{L}_1^{-1})$. Hence, the vortices on $X$ with respect to the line bundles $\mathcal{L}_0$ and $K_X\otimes\mathcal{L}_1^{-1}$ can be used to produce solutions of the Seiberg--Witten equations. Explicit examples are easy to find. We explain two classes of examples below. 
\begin{enumerate}
    \item\label{example 1} Take $3$ compact Riemann surfaces $(X_i,\omega_i)_{i=1,2,3}$ of the same genus $\textit{g}>1.$ $\omega_i$ denotes the normalized Kähler form on $X_i$ such that $\int_{X_i}\omega_i=1.$  Define $X:=X_1\times X_2\times X_3.$ The Kähler form on $X$ is $\omega:=\sum_{i=1}^{3}\pi_i^*\omega_i$ where $\pi_i$ is the projection of $X$ onto $X_i.$ Thereafter $K_X=\otimes_{i=1}^{3}\pi_i^*(K_{X_i}).$ If we take $\mathcal{L}_0=\underline{\mathbb{C}},$ the trivial line bundle and $\mathcal{L}_1=K_X$, these satisfy all the four necessary conditions:
\begin{align*}
    \text{For } \mathcal{L}_0=\underline{\mathbb{C}}, \text{ we have }H^0(X,\mathcal{L}_0)\cong \mathbb{C}\text{ and }c_1(K_X^{-1}\otimes \mathcal{L}_0^2)=-(2\textit{g}-2)\omega\\
    \text{For } \mathcal{L}_1=K_X, \text{ we have }H^0(X,K_X\otimes\mathcal{L}_1^{-1})\cong \mathbb{C}\text{ and }c_1(K_X^{-1}\otimes \mathcal{L}_1^2)=(2\textit{g}-2)\omega
\end{align*}
\item Another class of examples come from taking $X$ a hypersurface in $\mathbb{C}\mathbb{P}^4$ of very high degree, let's take a holomorphic section of $\mathcal{O}(d)\rightarrow \mathbb{C}\mathbb{P}^4,$ i.e., a homogeneous polynomial of degree $d$ in $5$ variables ($d$ to be determined later). If we choose this generically, the zero locus is a smooth algebraic variety $X$. The K\"ahler form $\omega$ on $X$ is given by restricting the Fubini-Study form $\omega_{FS}$ on $X$ and it lies in the cohomology class obtained by restricting $c_1(\mathcal{O}(1))$ on $X$. Meanwhile, by the adjunction formula, $K_X=\mathcal{O}(d-5)|_X$.\par
Now we take $\mathcal{L}_0=\mathcal{O}(k_0)|_X$ and $\mathcal{L}_1=\mathcal{O}(k_1)|_X$ ($k_0$ and $k_1$ to be determined later). For $m>0,$ dim $H^0(\mathbb{C}\mathbb{P}^4,\mathcal{O}(m))>0.$ Restricting these holomorphic sections to $X$ we can find line bundles on $X$ with non-trivial holomorphic sections. We need $k_0$ and $k_1$ to satisfy the following conditions:
\begin{align*}
    &1. \hspace{1 ex}k_0 >0 \text{ so that there are holomorphic sections of } \mathcal{L}_0.\\
    &2.\hspace{1 ex} (5-d)+2k_0 < 0. \text{ This ensures that } c_1(K_X^{-1}\otimes \mathcal{L}_0^2) \text{ is a negative multiple of } \omega.\\
    &3. \hspace{1 ex}d-5-k_1 >0 \text{ so that there are holomorphic sections of } K_X\otimes \mathcal{L}_1^{-1}.\\
    &4. \hspace{1 ex}(5-d)+2k_1 >0. \text{ This ensures that } c_1(K_X^{-1}\otimes \mathcal{L}_1^2) \text{ is a positive multiple of } \omega.
\end{align*}
Putting this conditions together we get
\begin{align*}
    0<k_0<\frac{d-5}{2}<k_1<d-5
\end{align*}
So, in this way we get many examples by first choosing $d>7$ and then picking $k_0$ and $k_1$ as we like in the above ranges.
\end{enumerate}\par
Notice in both the examples, $X$ happens to have ample canonical bundle, one might wonder if one can also produce solutions on Fano and Calabi-Yau manifolds. For this we look a ``perturbed" version of the Seiberg--Witten equations for $\phi\in\Gamma(S_+),\psi\in\Gamma(\tilde{S}_-),\beta\in\Omega^3,A\in\mathcal{A}$ and $B\in\mathcal{B}$.
\begin{align}
    \big(D_A+c(*\beta)\big)\phi=0\label{Dirac 6d + 100}\\
    F_A+2id^*\beta+2\pi ir_0\omega=q(\phi)\label{Curvature 6d + 100}\\
    \big(D_B+c(\beta)\big)\psi=0\label{Dirac 6d - 100}\\
    -{i}*F_B+2d\beta-\pi r_1\omega^2=q(\psi)\label{Curvature 6d - 100}
\end{align}
The equations are parametrized by two constants $r_0,r_1\in\mathbb{R}.$ Using the same reasoning as in the proof of theorem \ref{Theorem 6d}, we arrive at the following result. 
\begin{theorem}\label{SW 6d CY Fano}
    For a fixed pair of holomorphic structures on $\mathcal{L}_0$ and $\mathcal{L}_1$, there exists a non-trivial solution to the Seiberg--Witten equations \eqref{Dirac 6d + 100},\eqref{Curvature 6d + 100},\eqref{Dirac 6d - 100} and \eqref{Curvature 6d - 100} under the following conditions:
    \begin{align*}
    &1.\hspace{1 ex} \text{dim }H^0(X,\mathcal{L}_0)>0\\
    &2. \hspace{1 ex}c_1(K_X^{-1}\otimes \mathcal{L}_0^2)=a_0[\omega]\hspace{1 ex} \text{with $a_0<r_0$}\\
    &3. \hspace{1 ex} \text{dim }H^0(X,K_X\otimes\mathcal{L}_1^{-1})>0\\
    &4. \hspace{1 ex}c_1(K_X^{-1}\otimes \mathcal{L}_1^2)=a_1[\omega]\hspace{1 ex} \text{with $a_1>-r_1$}
\end{align*}
\end{theorem}
$r_0=r_1=0$ gives us theorem \ref{Theorem 6d}. The non-trivial solution established in the theorem satisfies $\phi\in\Omega^0(X,\mathcal L_0)$ and $\psi\in\Omega^{0,3}(X,\mathcal L_1)\cong \Omega^0(X,K_X^{-1}\otimes\mathcal{L}_1).$ For large enough $r_0$ and $r_1$, one can produce solutions to the equations \eqref{Dirac 6d + 100},\eqref{Curvature 6d + 100},\eqref{Dirac 6d - 100} and \eqref{Curvature 6d - 100} on Calabi-Yau and Fano three-folds and the moduli space of the constructed solutions in the theorem can be identified again using vortices on $X$ with respect to $\mathcal{L}_0$ and ${K_X\otimes\mathcal{L}_1^{-1}}.$\par
Too see an explicit example, take a holomorphic section of $\mathcal{O}(d)\rightarrow \mathbb{C}\mathbb{P}^4,$ i.e., a homogeneous polynomial of degree $d$ in $5$ variables ($d$ to be determined later). If we choose this generically, the zero locus is a smooth algebraic variety $X$. The K\"ahler form $\omega$ on $X$ is given by restricting the Fubini-Study form $\omega_{FS}$ on $X$ and it lies in the cohomology class obtained by restricting $c_1(\mathcal{O}(1))$ on $X$. Meanwhile, by the adjunction formula, $K_X=\mathcal{O}(d-5)|_X$.\par
Now we take $\mathcal{L}_0=\mathcal{O}(k_0)|_X$ and $\mathcal{L}_1=\mathcal{O}(k_1)|_X$ ($k_0$ and $k_1$ to be determined later). For $m>0,$ dim $H^0(\mathbb{C}\mathbb{P}^4,\mathcal{O}(m))>0.$ Restricting these holomorphic sections to $X$ we can find line bundles on $X$ with non-trivial holomorphic sections. Following theorem \ref{SW 6d CY Fano} we need $k_0$ and $k_1$ to satisfy the following conditions:
\begin{align*}
    &1. \hspace{1 ex}k_0 >0 \\
    &2.\hspace{1 ex} (5-d)+2k_0 < r_0\\
    &3. \hspace{1 ex}d-5-k_1 >0 \\
    &4. \hspace{1 ex}(5-d)+2k_1 >-r_1
\end{align*}
Hence for fixed $k_0,k_1,d$, if $k_0 >0,d-5-k_1 >0$ there exists $R>0,$ such that all four conditions are satisfied for $r_0,r_1>R.$ $d=5$ gives us Calabi-Yau three-folds and $d<5$ give us Fano three-folds.\par
In \S\ref{Riemann surface times C2} we produce solutions of another ``perturbed" version of the Seiberg--Witten equations on $\Sigma\times\mathbb{C}^2$ ($\Sigma$ being a closed Riemann surface). The solutions found are invariant in the $\mathbb{C}^2$ directions and can be described in terms of vortices on $\Sigma$.\par 
We denote the K\"ahler forms on $\Sigma$ and $\mathbb{C}^2$ by $\omega_\Sigma$ and $\omega_{\mathbb{C}^2}$ respectfully. The ``perturbed" Seiberg--Witten equations on $\Sigma\times\mathbb{C}^2$ are 
\begin{align}
    \big(D_A+c(*\beta)\big)\phi=0\label{Dirac 6d + sigma}\\
    F_A+2id^*\beta+ir_0\pi_2^*(\omega_{\mathbb{C}^2})=q(\phi)\label{Curvature 6d + sigma}\\
    \big(D_B+c(\beta)\big)\psi=0\label{Dirac 6d - sigma}\\
    -{i}*F_B+2d\beta-r_1*\big(\pi_2^*(\omega_{\mathbb{C}^2})\big)=q(\psi)\label{Curvature 6d - sigma}
\end{align}
The equations are parametrized by two real constants $r_0,r_1.$ $\pi_1,\pi_2$ denote the standard projection maps $\pi_1:\Sigma\times\mathbb{C}^2\rightarrow \Sigma$ and $\pi_2:\Sigma\times\mathbb{C}^2\rightarrow\mathbb{C}^2.$\par
The spin bundles on $\mathbb{C}^2$ are $S_+(\mathbb{C}^2)=\Lambda^0\oplus\Lambda^{0,2},S_-(\mathbb{C}^2)=\Lambda^{0,1}.$ For two hermitian holomorphic line bundles $\mathcal{L}_0$ and $\mathcal{L}_1,$ we define two spin bundles of opposite chirality on $\Sigma\times\mathbb{C}^2:$
\begin{align*}
    S_+(\Sigma\times\mathbb{C}^2)=\Big(\pi_1^*(\mathcal{L}_0)\otimes \pi_2^*\big(S_+(\mathbb{C}^2)\big)\Big)\oplus \Big(\pi_1^*(K_\Sigma^{-1}\otimes\mathcal{L}_0)\otimes \pi_2^*\big(S_-(\mathbb{C}^2)\big)\Big)\\
    \tilde{S}_-(\Sigma\times\mathbb{C}^2)=\Big(\pi_1^*(\mathcal{L}_1)\otimes \pi_2^*\big(S_-(\mathbb{C}^2)\big)\Big)\oplus \Big(\pi_1^*(K_\Sigma^{-1}\otimes\mathcal{L}_1)\otimes \pi_2^*\big(S_+(\mathbb{C}^2)\big)\Big)
\end{align*}
We prove the following.
\begin{theorem}\label{Theorem 6d sigma}
    For a fixed pair of holomorphic structures on $\mathcal{L}_0$ and $\mathcal{L}_1$, there exists a non-trivial solution to the Seiberg--Witten equations \eqref{Dirac 6d + sigma},\eqref{Curvature 6d + sigma},\eqref{Dirac 6d - sigma},\eqref{Curvature 6d - sigma} under the following conditions:
    \begin{align*}
    &1.\hspace{1 ex} \text{dim }H^0(\Sigma,\mathcal{L}_0)>0\\
    &2.\hspace{1 ex}r_0=\frac{2\pi}{\text{vol}(\Sigma)}\text{deg}(K_\Sigma-2\mathcal{L}_0)>0\\
    &3.\hspace{1 ex} \text{dim }H^0(\Sigma,K_\Sigma-\mathcal{L}_1)>0\\
    &4.\hspace{1 ex} r_1=-\frac{2\pi}{\text{vol}(\Sigma)}\text{deg}(K_\Sigma-2\mathcal{L}_1)>0
\end{align*}
\end{theorem}
The solutions found in theorem \ref{Theorem 6d sigma} are invariant in the $\mathbb{C}^2$ directions, and modulo gauge the space of the solutions found can be identified with $\mathcal{M}_{\text{vortex}}(\mathcal{L}_0)\times \mathcal{M}_{\text{vortex}}(K_\Sigma-\mathcal{L}_1)$. Hence, vortices on $\Sigma$ with respect to the line bundles $\mathcal{L}_0$ and $K_\Sigma-\mathcal{L}_1$ can be used to find solutions to the Seiberg--Witten equations \eqref{Dirac 6d + sigma},\eqref{Curvature 6d + sigma},\eqref{Dirac 6d - sigma},\eqref{Curvature 6d - sigma}.\par
Explicit examples are easy to find. One can start with any compact Riemann surface $\Sigma$ with genus $\textit{g}>1.$ Then for $\mathcal{L}_0=\underline{\mathbb{C}}$, the trivial line bundle and $\mathcal{L}_1=K_\Sigma,$ the canonical line bundle and $r_0=r_1=\frac{2\pi}{\text{vol}(\Sigma)}(2\textit{g}-2)$, the four necessary conditions are satisfied.

In \S\ref{8d} we construct solutions of the $8$-dimensional Seiberg--Witten equations on a closed K\"ahler $4$-fold $(X,\omega).$ On a Riemannian $8$-manifold, the hodge-star operator $*$ squares to $1$ on $\Lambda^4$. Hence it induces a splitting of the four forms into self-dual and anti-self dual four forms: $\Lambda^4=\Lambda^{4}_+\oplus\Lambda^{4}_-.$ In dimension $8,$ Clifford multiplication gives an isomorphism:
\begin{align}\label{Clifford 8d}
    c:i\Lambda^2\oplus\Lambda^{4}_+\rightarrow i\mathfrak{su}(S_+)
\end{align}
Given $\phi\in\Gamma(S_+),$ we write $q(\phi)$ for the corresponding differential form $E_\phi$ under \eqref{Clifford 8d}.\par 
We define the ``perturbed" Seiberg--Witten equations on $X$ for $(\phi,A,\beta),\phi\in\Gamma(S_+),A\in\mathcal{A},\beta\in\Omega^3:$
\begin{align}
   \big(D_A+(1+i)c(\beta)\big)\phi=0\label{Dirac 8d}\\
   F_A+2d\beta^++2id^*\beta+a\omega^2=q(\phi)\label{Curvature 8d}
\end{align}
$a\in\mathbb{R}$ and $d\beta^+$ denotes the self-dual part of $d\beta$. Notice that the perturbation term: $a\omega^2$ is \textit{harmonic} and can be thought of as the cohomology class of $d\beta^+.$ $a=0$ gives us back the $8$-dimensional Seiberg--Witten equations defined in \cite{JP}. The gauge group $\mathcal{G}=$ Map$(X,S^1)$ acts on $(\phi,A)$ by pull back and trivially on $\beta.$ The equations are elliptic modulo gauge.\par
For a hermitian holomorphic line bundle $\mathcal{L}\rightarrow X,$ one can take
\begin{align*}
    S_+(X)=\Lambda^0(X,\mathcal{L})\oplus \Lambda^{0,2}(X,\mathcal{L})\oplus \Lambda^{0,4}(X,\mathcal{L}).
\end{align*}
We prove the following theorem in \S\ref{Solution 8d}.
\begin{theorem}\label{theorem 8d}
    For a fixed choice of holomorphic structure on $\mathcal{L},$ there exists a non-trivial solution to the Seiberg--Witten equations \eqref{Dirac 8d},\eqref{Curvature 8d} under the following conditions:
    \begin{align*}
    &1.\hspace{1 ex} \text{dim }H^0(X,\mathcal{L})>0\\
    &2. \hspace{1 ex} a<0 \text{ and }c_1(K_X^{-1}\otimes \mathcal{L}^2)=\frac{2a}{\pi}[\omega]\hspace{1 ex}\text{[this implies deg $(K_X^{-1}\otimes \mathcal{L}^2)<0$]}\\
    &\hspace{32 ex}\text{or}\\
    &1. \hspace{1 ex} \text{dim }H^0(X,K_X\otimes\mathcal{L}^{-1})>0\\
    &2. \hspace{1 ex} a<0 \text{ and }c_1(K_X^{-1}\otimes \mathcal{L}^2)=-\frac{2a}{\pi}[\omega]\hspace{1 ex}\text{[this implies deg $(K_X^{-1}\otimes \mathcal{L}^2)>0$]}
\end{align*}
\end{theorem}
If the first pair of conditions are satisfied, the non-trivial solution established in the theorem satisfies $\phi\in\Omega^0(X,\mathcal L)$, and if the second pair of conditions are satisfied, $\phi\in\Omega^{0,4}(X,\mathcal L)\cong \Omega^0(X,K_X^{-1}\otimes\mathcal{L})$. Depending on the pair of conditions satisfied, the moduli space of solutions obtained in theorem \ref{theorem 8d} can be identified with \(\mathcal{M}_{\text{vortex}}(\mathcal{L})\) (when deg$(K_X\otimes\mathcal{L}^2)<0)$ or \(\mathcal{M}_{\text{vortex}}(K_X\otimes\mathcal{L}_1^{-1})\) (when deg$(K_X\otimes\mathcal{L}^2)>0)$. Hence, vortices on \(X\) with respect to the line bundle \(\mathcal{L}\) or \(K_X \otimes \mathcal{L}^{-1}\) can be used to find solutions to the Seiberg--Witten equations \eqref{Dirac 8d},\eqref{Curvature 8d}. Explicit examples are easy to find. We describe two classes of solutions below. 
\begin{enumerate}
    \item  Take $4$ compact Riemann surfaces $(X_i,\omega_i)_{i=1,2,3,4}$ of the same genus $g>1.$ $\omega_i$ denotes the normalized Kähler form on $X_i$ such that $\int_{X_i}\omega_i=1.$  Define $X:=X_1\times X_2\times X_3\times X_4.$ The Kähler form on $X$ is $\omega:=\sum_{i=1}^{4}\pi_i^*\omega_i$ where $\pi_i$ is the projection of $X$ onto $X_i.$ We have $K_X=\otimes_{i=1}^4\pi_i^*(K_{X_i})$ and $c_1(K_X)=(2g-2)\omega.$\par 
    Take $a=(1-g)\pi$. If we take $\mathcal{L}$ to be the trivial line bundle it satisfies the first two conditions and if we take $\mathcal{L}=K_X,$ it satisfies the alternative two conditions.
    \item We take $X$ to be a hypersurface in $\mathbb{C}\mathbb{P}^4$ of very high degree, let's take a holomorphic section of $\mathcal{O}(d)\rightarrow \mathbb{C}\mathbb{P}^5,$ i.e., a homogeneous polynomial of degree $d$ in $6$ variables ($d$ to be determined later). If we choose this generically, the zero locus is a smooth algebraic variety $X$.\par
The Kahler form $\omega$ on $X$ is given by restricting the Fubini-Study form $\omega_{FS}$ on $X$ and it lies in the cohomology class obtained by restricting $c_1(\mathcal{O}(1))$ on $X$. Meanwhile, by the adjunction formula, $K_X=\mathcal{O}(d-6)|_X$.\par
Take $\mathcal{L}=\mathcal{O}(k)|_X$. For $m>0,$ dim $H^0(\mathbb{C}\mathbb{P}^5,\mathcal{O}(m))>0.$ Restricting these holomorphic sections to $X$ we can find line bundles on $X$ with non-trivial holomorphic sections. For the existence of non-trivial solutions, we need $a,k$ and $d$ to satisfy:
\begin{align*}
    &1. \hspace{1 ex}k>0, \text{ for dim } H^0(X,\mathcal L_0)>0.\\
    &2. \hspace{1 ex}a=\frac{\pi}{2}(6-d+2k)<0. \text{ This ensures that } a<0 \text{ and }c_1(K_X^{-1}\otimes \mathcal{L}^2)=\frac{2a}{\pi}[\omega].
\end{align*}
or
\begin{align*}
    &1. \hspace{1 ex}d-6-k>0, \text{ for dim } H^0(X,K_X\otimes\mathcal{L}^{-1})>0.\\
    &2. \hspace{1 ex}a=-\frac{\pi}{2}(6-d+2k)<0. \text{ This ensures that } a<0 \text{ and }c_1(K_X^{-1}\otimes \mathcal{L}^2)=-\frac{2a}{\pi}[\omega].
\end{align*}
Putting these all together we would need $k$ to be in the following ranges:
\begin{align*}
    &0<k<\frac{d-6}{2}\hspace{5 ex}\text{ or } 
    \hspace{5 ex}\frac{d-6}{2}<k<d-6
\end{align*}
So we get many examples by choosing $d>8$ and choosing $k$ as we like in the above ranges.
\end{enumerate}
In both the examples $X$ has ample canonical line bundle. To produce solutions on Calabi-Yau and Fano four-folds, we look at another ``perturbed" version of the Seiberg--Witten equations parametrized by three real constants $r_0,r_1,a.$ For $\phi\in\Gamma(S_+),A\in\mathcal{A},\beta\in\Omega^3$, the equations are:
\begin{align}
   \big(D_A+(1+i)c(\beta)\big)\phi=0\label{Dirac 8d new}\\
   F_A+2d\beta^++2id^*\beta+4ir_0\omega-r_1\omega^2+a\omega^2=q(\phi)\label{Curvature 8d new}
\end{align}
Using the same reasoning as in the proof of theorem \ref{theorem 8d}, we arrive at the following result.
\begin{theorem}\label{theorem 8d new}
    For a fixed choice of holomorphic structure on $\mathcal{L},$ there exists a non-trivial solution to the Seiberg--Witten equations \eqref{Dirac 8d new},\eqref{Curvature 8d new} under the following conditions:
    \begin{align*}
    &1.\hspace{1 ex} \text{dim }H^0(X,\mathcal{L})>0\\
    &2. \hspace{1 ex} r_0=r_1>a \text{ and }c_1(K_X^{-1}\otimes \mathcal{L}^2)=\frac{2a}{\pi}[\omega]\hspace{1 ex}\\
    &\hspace{16 ex}\text{or}\\
    &1. \hspace{1 ex} \text{dim }H^0(X,K_X\otimes\mathcal{L}^{-1})>0\\
    &2. \hspace{1 ex} -r_0=r_1>a \text{ and }c_1(K_X^{-1}\otimes \mathcal{L}^2)=-\frac{2a}{\pi}[\omega]\hspace{1 ex}
\end{align*}
\end{theorem}
If the first pair of conditions are satisfied, the non-trivial solution established in the theorem satisfies $\phi\in\Omega^0(X,\mathcal L)$, and if the second pair of conditions are satisfied, $\phi\in\Omega^{0,4}(X,\mathcal L)\cong \Omega^0(X,K_X^{-1}\otimes\mathcal{L})$. To see an example, we start with a holomorphic section of $\mathcal{O}(d)\rightarrow \mathbb{C}\mathbb{P}^5,$ i.e., a homogeneous polynomial of degree $d$ in $6$ variables ($d$ to be determined later). If we choose this generically, the zero locus is a smooth algebraic variety $X$.\par
The Kahler form $\omega$ on $X$ is given by restricting the Fubini-Study form $\omega_{FS}$ on $X$ and it lies in the cohomology class obtained by restricting $c_1(\mathcal{O}(1))$ on $X$. Meanwhile, by the adjunction formula, $K_X=\mathcal{O}(d-6)|_X$.\par
Take $\mathcal{L}=\mathcal{O}(k)|_X$. For $m>0,$ dim $H^0(\mathbb{C}\mathbb{P}^5,\mathcal{O}(m))>0.$ Restricting these holomorphic sections to $X$ we can find line bundles on $X$ with non-trivial holomorphic sections. We need $k,d,r_0,r_1,a$ to satisfy the following two conditions:
\begin{align*}
    &1. \hspace{1 ex}k>0\\
    &2. \hspace{1 ex}r_0=r_1>a=\frac{\pi}{2}(6-d+2k)
\end{align*}
And for $\mathcal{O}(k)|_X$ to satisfy the alternative two conditions we would need:
\begin{align*}
    &1. \hspace{1 ex}d-6-k>0\\
    &2. \hspace{1 ex}-r_0=r_1>a=-\frac{\pi}{2}(6-d+2k)
\end{align*}
Therefore, for a fixed choice of $k$ and $d$ with $k>0$ (or $d-6-k>0$), there exists an $R>0$, such that for all $r_0=r_1>R$ (or $-r_0=r_1>R$) and $a=\frac{\pi}{2}(2k+6-d)$ (or $a=-\frac{\pi}{2}(2k+6-d)$), we have non-trivial solutions to the Seiberg--Witten equations \eqref{Dirac 8d new},\eqref{Curvature 8d new}. For $d=6,X$ is a Calabi-Yau four-fold and for $d<6,X$ is a Fano four-fold.\par 
In \cite{JP} we constructed a solution to the  $5$-dimensional Seiberg--Witten equations on the total
space of the unit circle bundle of an ample line bundle of a K\"ahler Einstein surface. Another interesting class of five-dimensional manifolds is $\Sigma\times\mathbb{R}^3,$ where $\Sigma$ is a compact Riemann surface without boundary. For the $3$-dimensional Seiberg--Witten equations, it is known that for the product metric on $\Sigma\times S^1,$ the solutions of the Seiberg--Witten equations turn out to be invariant under the $S^1$-action and the moduli space of solutions of the Seiberg--Witten equations can be identified with the vortices over $\Sigma$ \cite{Salamon}. This correspondence motivates the exploration of solutions to the $5$-dimensional Seiberg-Witten equations on $\Sigma\times\mathbb{R}^3$ with the additional condition that these solutions remain invariant under translations along the $\mathbb{R}^3$ directions.\par
In \S\ref{5d} we construct solutions of the $5$-dimensional Seiberg--Witten equations on $\Sigma\times\mathbb{R}^3,$ where $\Sigma$ has genus greater than $1.$ The solutions are invariant in the $\mathbb{R}^3$ directions. Moreover the space of constructed solutions can be identified with the vortices over $\Sigma.$ We explain the result in details below.\par
In dimension $5,$ Clifford multiplication gives an isomorphism:
\begin{align}\label{Clifford 5d}
    c:i\Lambda^2\oplus\Lambda^4\rightarrow i\mathfrak{su}(S)
\end{align}
Given $\phi\in\Gamma(S),$ we write $q(\phi)$ for the corresponding differential form $E_\phi$ under \eqref{Clifford 5d}. $\mathbb{R}^3$ has the standard Euclidean metric, with $x_1,x_2,x_3$ denoting the coordinate directions. Let $\omega$ denotes the K\"ahler form on $\Sigma.$ We define the ``perturbed" Seiberg--Witten equations on $\Sigma\times\mathbb{R}^3$ for $(\phi,A,\beta_j),\phi\in\Gamma(S),A\in\mathcal{A},\beta_j\in\Omega^j:$
\begin{align}
\big(D_A+c((1-i)\beta_3+i*\beta_5)\big)\phi=0\label{Dirac 5D perturbed}\\
F_A-2id^*\beta_3+2d\beta_3+2d^*\beta_5+\eta=q(\phi)\label{Curvature 5D perturbed}\\ 
\nonumber
\eta= a(id x_1\wedge d x_2\pm\omega\wedge d x_1\wedge d x_2),a\in\mathbb{R}     
\end{align}

The perturbation term $\eta$ is \textit{harmonic} and we can think of $\pm a\hspace{0.2 ex}\omega\wedge dx_1\wedge dx_2$ as playing the
role of the cohomology class of $2d\beta_3+2d^*\beta_5.$ $a=0$ gives us back the $5$-dimensional Seiberg--Witten equations defined in \cite{JP}.\par
The gauge group $\mathcal{G}=$ Map$(\Sigma\times\mathbb{R}^3,S^1)$ acts on $A$ and $\phi$ by pull-back and trivially on $\beta_3,\beta_5.$ This action preserves the above equations and the equations form an elliptic system modulo the gauge action \cite{JP}.\par
Let $\mathcal{L}\rightarrow\Sigma$ be a hermitian holomorphic line bundle over $\Sigma$ and say $\underline{\mathbb{C}}^2$ denotes the trivial complex vector bundle of rank $2$ over $\mathbb{R}^3$ (with the standard flat metric). If we take the product metric on $\Sigma\times\mathbb{R}^3,$ the spinor bundle $S\rightarrow\Sigma\times\mathbb{R}^3$ can be taken as:
\begin{align}\label{Spinor in 5d}
    S(\Sigma\times\mathbb{R}^3)=\big(\pi_1^*(\Lambda^0(\Sigma,\mathcal{L}))\otimes\pi_2^*(\underline{\mathbb{C}}^2)\big)\oplus \big(\pi_1^*(\Lambda^{0,1}(\Sigma,\mathcal{L}))\otimes\pi_2^*(\underline{\mathbb{C}}^2)\big)
\end{align}
$\pi_1,\pi_2$ denote the standard projection maps $\pi_1:\Sigma\times\mathbb{R}^3\rightarrow \Sigma$ and $\pi_2:\Sigma\times\mathbb{R}^3\rightarrow\mathbb{R}^3.$ We prove the following in \S\ref{Soln 5d}. 
\begin{theorem}\label{Theorem 5d}
For a fixed choice of holomorphic structure on $\mathcal{L}$, there exists a non-trivial solution to the  Seiberg--Witten equations \eqref{Dirac 5D perturbed} and \eqref{Curvature 5D perturbed} under the following conditions:
\begin{align}
    &1.\hspace{1 ex} \text{dim }H^0(X,\mathcal{L})>0, \text{deg}(K_\Sigma-2\mathcal{L})>0\label{condition 1 5d}\\
    &2. \hspace{1 ex} a=\pm\frac{2\pi}{\text{vol}(\Sigma)}\text{deg}(K_\Sigma-2\mathcal{L})\\
    &\hspace{24 ex}\text{or}\nonumber\\
    &1. \hspace{1 ex} \text{dim }H^0(X,K_X\otimes\mathcal{L}^{-1})>0, \text{deg}(K_\Sigma-2\mathcal{L})<0\label{condition 2 5d}\\
    &2. \hspace{1 ex} a=\pm\frac{2\pi}{\text{vol}(\Sigma)}\text{deg}(K_\Sigma-2\mathcal{L})
\end{align}
The solution is translation invariant in any direction in $\mathbb{R}^3.$
\end{theorem}
\begin{remark}
    For $a\neq 0$, both the conditions mentioned in theorem \ref{Theorem 5d} imply that genus$(\Sigma)>1.$ Condition \ref{condition 1 5d} gives us: $0<\text{ deg}(\mathcal{L})<$ genus$(\Sigma)-1.$ Thus genus$(\Sigma)>1$. From condition \ref{condition 2 5d} we get genus$(\Sigma)-1<$ deg$(\mathcal{L})<2$ genus$(\Sigma)-2$, so genus$(\Sigma)>1$. Hence, it is a necessary condition on $\Sigma$ for our construction to work.
\end{remark}
Relationship of these solutions with vortices is case specific. For deg$(K_\Sigma-2\mathcal{L})>0$, the space of constructed solutions modulo gauge can be identified as $\mathcal{M}_{\text{vortex}}(\mathcal{L})$ and for deg$(K_\Sigma-2\mathcal{L})<0$, the space of constructed solutions modulo gauge can be identified as $\mathcal{M}_{\text{vortex}}(K_\Sigma-\mathcal{L}).$ Explicit examples are again not difficult to find. One can take any compact Riemann surface $\Sigma$ with genus $\textit{g}>1.$ Now if $\mathcal{L}$ is the trivial line bundle, then for $a=\pm \frac{4\pi(\textit{g}-1)}{\text{vol}(\Sigma)},$ the first two conditions in theorem \ref{Theorem 5d} are satisfied. If we take $\mathcal{L}=K_\Sigma$ and $a=\pm\frac{4\pi(\textit{g}-1)}{\text{vol}(\Sigma)},$ it satisfies the two alternative conditions in theorem \ref{Theorem 5d}.\par
\vspace{1 ex}
\textbf{Acknowledgments.}
This article comes from my Ph.D. thesis. I sincerely thank my thesis supervisor Joel Fine, for suggesting this project, providing continuous support and guidance, and offering valuable feedback on an initial draft of this article. During the preparation of this article, I was supported by FNRS grants FRIA B1-$40004195$ and FRIA B2-$40014665.$
\section{Solutions to the $6$-dimensional Seiberg--Witten equations}\label{6d}
This section is dedicated to constructing solutions of the $6$-dimensional Seiberg--Witten equations on K\"ahler manifolds. We construct solutions on closed K\"ahler $3$-folds and on $\Sigma\times\mathbb{C}^2 \hspace{.4 ex}(\Sigma$ being a closed Riemann surface).  
\subsection{Spin$^\mathbb{C}$-bundle, Dirac operator and Clifford multiplication}
Let $(X,\omega)$ be a K\"ahler $3$-fold. The K\"ahler form $\omega$ together with a Riemannian metric determines a unique compatible almost complex structure $J.$ For two holomorphic hermitian line bundles $\mathcal{L}_0,\mathcal{L}_1\rightarrow X,$ we have positive and negative spin bundles:
\begin{align*}
    S_+=(\Lambda^0\oplus\Lambda^{0,2})\otimes\mathcal{L}_0\\
    \tilde{S}_-=(\Lambda^{0,1}\oplus\Lambda^{0,3})\otimes\mathcal{L}_1
\end{align*}
The associated line bundles are
\begin{align*}
    L(S)=K_X^{-1}\otimes\mathcal{L}_0^2\\
    L(\tilde{S})=K_X^{-1}\otimes\mathcal{L}_1^2
\end{align*}
A unitary connection $A$ on $L(S)$ is equivalent to a unitary connection $A_0$ on $\mathcal{L}_0,$ the equivalence being $A=-A_{K_X}+2A_0$ (with abuse of notation) where $A_{K_X}$ is the Chern connection on $K_X$ (the metric on $K_X$ is the one induced from the K\"ahler metric). Similarly, a unitary connection $B_0$ on $\mathcal{L}_1$ gives us a unitary connection $B$ on $L(\tilde{S})$ via $B=-A_{K_X}+2B_0$ (with abuse of notation again).\par 
The Dirac operators associated to the connection $A$ on $\mathcal{L}$ and the connection $B$ on $\tilde{\mathcal{L}}$ respectively on $\Gamma(S_+)$ and $\Gamma(\tilde{S}_-)$ are \cite{Morgan}
\begin{align*}
    D_A=\sqrt{2}(\bar\partial_{A_0}+{\bar\partial}_{A_0}^*):\Omega^0(X,\mathcal{L}_0)\oplus \Omega^{0,2}(X,\mathcal{L}_0)\rightarrow \Omega^{0,1}(X,\mathcal{L}_0)\oplus \Omega^{0,3}(X,\mathcal{L}_0)\\
    D_B=\sqrt{2}(\bar\partial_{B_0}+{\bar\partial}_{B_0}^*):\Omega^{0,1}(X,\mathcal{L}_1)\oplus \Omega^{0,3}(X,\mathcal{L}_1)\rightarrow \Omega^{0}(X,\mathcal{L}_1)\oplus \Omega^{0,2}(X,\mathcal{L}_1).
\end{align*}
The operators are obtained by coupling $\sqrt{2}(\bar\partial+{\bar\partial}^*)$ with the covariant derivatives $\nabla_{A_0}$ on $\mathcal{L}_0$ and $\nabla_{A_1}$ on $\mathcal{L}_1.$\par
The almost complex structure 
$J$ induces a splitting of complexified $k$-forms: $\Lambda^k(X)\otimes\mathbb{C}=\sum_{p+q=k}\Lambda^{p,q}(X).$ Define $\Omega^3_{\pm}:=\Gamma(\Lambda^3_{\pm}).$ In the following lemma we investigate the relationship between this decomposition and the Lefschetz decomposition of a $3$-form. This will be useful in understanding the Clifford action of a real-valued $3$-form on positive and negative spinors.
\begin{lemma}\label{lemma Lefschetz}
    On a Kähler $3$-fold $(X,\omega),$ any real valued $3$-form $\beta$ can be written as
    \begin{align*}
        \beta=\beta^{3,0}_-+(\eta\wedge\omega)_++\gamma_-+\overline{\beta^{3,0}}_++(\bar\eta\wedge\omega)_-+\bar{\gamma}_+
    \end{align*}
    Where $\beta^{3,0}\in\Omega^{3,0}(X,\mathbb{C}),\eta\in\Omega^{0,1}(X,\mathbb{C}),\gamma\in$\text{ Ker}$\big(\wedge\omega:\Omega^{1,2}(X,\mathbb{C})\rightarrow\Omega^{2,3}(X,\mathbb{C})\big)$. The subscript $'+'$ or $'-'$ respectively say if the form lies in $\Omega^3_+$ or $\Omega^3_-$. 
\end{lemma}
\begin{proof}
    Lefschetz decomposition says that $\beta$ can be written as 
    \begin{align*}
        \beta=\beta^{3,0}_-+(\eta\wedge\omega)_++\gamma_-+\overline{\beta^{3,0}}_++(\bar\eta\wedge\omega)_-+\bar{\gamma}_+
    \end{align*}
    for some $\beta^{3,0}\in\Omega^{3,0},\eta\in\Omega^{0,1}$ and $\gamma\in$ Ker$(\wedge\omega:\Omega^{1,2}\rightarrow\Omega^{2,3}).$
    We take local holomorphic coordinates $\{z_k=x_k+iy_k\}_{k=1,2,3}$ centered at a point $x\in X$ so that the Kähler metric is standard to second order at the point. We start with $(1,2)$ forms. In local coordinates we see that at $x,$ $\{d z_k\wedge d\bar{z}_j\wedge d\bar{z}_l\}_{\{j\neq l\}}$ span the $(1,2)$ forms.\par 
We take $j,k,l\in\{1,2,3\}$ and $j\neq k\neq l$ for the rest of the proof. Observe that
\begin{align*}
    *(d\bar{z}_j\wedge d x_k\wedge d y_k)=i (d\bar{z}_j\wedge d x_l\wedge d y_l) 
\end{align*}
Hence
\begin{align*}
    *(d\bar z_j\wedge\omega)=*(d\bar z_j\wedge dx_k\wedge dy_k+dx_l\wedge dy_l)=i(d\bar z_j\wedge\omega)
\end{align*}
So, $\eta\wedge\omega\in\Omega^3_+.$ Notice that at $x,$ $d\bar z_j\wedge (dx_k\wedge dy_k-dx_l\wedge dy_l),d z_j\wedge d\bar{z}_k\wedge d\bar{z}_l$ span Ker$(\wedge\omega:\Omega^{1,2}\rightarrow\Omega^{2,3}).$ A simple calculation shows
\begin{align*}
    *(d z_j\wedge d\bar{z}_k\wedge d\bar{z}_l)=-i(d z_j\wedge d\bar{z}_k\wedge d\bar{z}_l)
\end{align*} 
Since, $*(d\bar{z}_j\wedge d x_k\wedge d y_k)=i (d\bar{z}_j\wedge d x_l\wedge d y_l),$ we have
\begin{align*}
    *(dx_k\wedge dy_k-dx_l\wedge dy_l)=-i(dx_k\wedge dy_k-dx_l\wedge dy_l)
\end{align*}
Thereafter, $\gamma\in\Omega^3_-.$
Similar observations for $(2,1)$ forms give us: $\bar\eta\wedge\omega\in\Omega^3_+$ and $\bar\gamma\in\Omega^3_-.$ Finally for $(3,0)$ and $(0,3)$ forms we notice
\begin{align*}
    *(d\bar z_j\wedge d\bar z_k\wedge d\bar z_l)=i(d\bar z_j\wedge d\bar z_k\wedge d\bar z_l)\\
    *(dz_j\wedge dz_k\wedge dz_l)=-i(dz_j\wedge dz_k\wedge dz_l)
\end{align*}
Hence the lemma follows.
\end{proof}
Next, we provide explicit formulae for the Clifford action of forms of the type \(\beta = (\eta + \bar{\eta}) \wedge \omega\), where \(\eta \in \Omega^{0,1}\). These formulae will be instrumental in constructing solutions to the $6$-dimensional Seiberg--Witten equations.\par 
Let $\alpha\in\Omega^1\otimes\mathbb{C}$ be a complexified one-form and $\xi\in\Omega^{0,k}$ be a spinor. Clifford action of $\alpha=(\alpha^{0,1}+\alpha^{1,0})$ on $\xi$ is given by the following formula \cite{Morgan} :
\begin{align}\label{Clifford one form}
c(\alpha)\xi=\sqrt{2}(\alpha^{0,1}\wedge\xi-\overline{\alpha^{1,0}}\angle\xi) 
\end{align}
$\overline{\alpha^{1,0}}\angle\xi\in\Omega^{0,k-1}$ is the contraction of $\alpha$ and $\xi$ defined by $\overline{\alpha^{1,0}}\angle\xi:=\iota_{(\overline{\alpha^{1,0}})^{\#}}(\xi)$, where $(\overline{\alpha^{1,0}})^{\#}$ is the vector metric dual to $\overline{\alpha^{1,0}}$. At a point $p\in X,$ if $\xi=e_1\wedge\cdots\wedge e_k$
\begin{align*}
     \overline{\alpha^{1,0}}\angle(e_1\wedge\cdots\wedge e_k):=\sum_{i=1}^k (-1)^{i-1}\langle e_i,\overline{\alpha^{1,0}}\rangle e_1\wedge\cdots\wedge\cdots\wedge\hat{e_i}\wedge\cdots\wedge e_k
\end{align*}
\begin{remark}\label{remark Clifford action}
    Notice $c(\alpha)$ takes positive spinors: $\oplus_p\Omega^{0,2p}$ to negative spinors: $\oplus_p\Omega^{0,2p+1}$ and vice versa. Moreover it says that for a $(p,q)$ form $\kappa, c(\kappa)$ takes a spinor in $\Omega^{0,k}$ to $\Omega^{0,k+q-p}.$ Hence the action is trivial if $k+q-p<0.$
\end{remark}
Before we proceed, we clarify some notations which will be used in the next part of the article. For a form in $\Omega^{p,q}(X,\mathcal{L}_0)$ we define $*:\Omega^{p,q}(X,\mathcal{L}_0)\rightarrow \Omega^{n-q,n-p}(X,\mathcal{L}_0),$ which is defined as the following 
\begin{align*}
    *(\phi\otimes s):=*(\phi)\otimes s \end{align*}
$\phi\in\Omega^{p,q}(X,\mathbb{C})$ and $s$ is a non-vanishing section of $\mathcal{L}_0$ defined locally, $*$ is complex-linear on $\Omega^{p,q}(X,\mathbb{C}).$\par
Notice that this is different from the usual $\bar*_h$ operator, which is defined using the hermitian metric (say $h$) on $\mathcal{L}_0.$ First we interpret $h$ as a $\mathbb{C}$-antilinear isomorphism $h:\mathcal{L}_0\cong\mathcal{L}_0^*.$ Then $\bar *_h$ is defined as $\bar*_h:\Omega^{p,q}(X,\mathcal{L}_0)\rightarrow \Omega^{n-p,n-q}(X,\mathcal{L}_0^*)$
\begin{align*}
    \bar*_h(\phi\otimes s):=\bar*(\phi)\otimes h(s)=\overline{*(\phi)}\otimes h(s)=*(\bar\phi)\otimes h(s)
\end{align*}
\begin{lemma}\label{Clifford action formulae} Let's take $\eta\in\Omega^{0,1},\gamma\in\text{Ker}(\wedge\omega:\Omega^{1,2}\rightarrow\Omega^{2,3}),\varphi\in\Omega^0(X,\mathcal{L}_0)$ and $\xi\in\Omega^{0,3}(X,\mathcal{L}_1).$ Then we have
\begin{align*}
    &1.\hspace{1 ex}c(\eta\wedge\omega)\varphi=-2\sqrt{2}i\eta\wedge\varphi\\
    &2.\hspace{1 ex} c(\bar\eta\wedge\omega)\xi=2\sqrt{2}*(\bar\eta\wedge\xi)\\
    &3. \hspace{1 ex} c(\bar\gamma)\varphi=0\\
    &4. \hspace{1 ex} c(\gamma)\xi=0
\end{align*}
\end{lemma}
\begin{proof}
    We take local holomorphic coordinates $\{z_k=x_k+iy_k\}_{k=1,2,3}$ centered at a point $x\in X$ so that the Kähler metric is standard to second order at the point. For rest of the proof we take $j,k,l\in\{1,2,3\},j\neq k\neq l.$ All the calculations done in this proof are at this point $x$ with respect to the chosen local coordinates.\par
    Using formula \ref{Clifford one form} we compute
    \begin{align*}
        c(d\bar z_j\wedge dx_k\wedge dy_k)\varphi&=c(d\bar z_j)\circ c( dx_k)\circ c( dy_k)\varphi\\
        &=c(d\bar z_j)\circ c( dx_k)(\frac{i}{\sqrt{2}}d\bar z_k\wedge\varphi)\\
        &=c(d\bar z_j)(-\frac{i}{2}|d\bar z_k|^2\varphi)\\
        &=-\sqrt{2}i d\bar z_j\wedge\varphi
    \end{align*}
    \begin{align*}
     \text{Hence, }c(d\bar z_j\wedge\omega)\varphi&=c(d\bar z_j\wedge dx_k\wedge dy_k+ d\bar z_j\wedge dx_l\wedge dy_l)\varphi=-2\sqrt{2}i d\bar z_j\wedge\varphi
    \end{align*}
    Thereafter, the first formula of the lemma follows since Clifford action is linear in $\eta.$ For the proof of the second formula, enough to take $\bar\eta=dz_j,\xi=d\bar z_j\wedge d\bar z_k\wedge d\bar z_l$ (at $x$) since Clifford action is linear in both $\bar\eta$ and $\xi.$ We compute
    \begin{align*}
        c(dz_j\wedge dx_k\wedge dy_k)(d\bar z_j\wedge d\bar z_k\wedge d\bar z_l)&=c(dz_j)\circ c(dx_k)\circ c( dy_k)(d\bar z_j\wedge d\bar z_k\wedge d\bar z_l)\\
        &=c(dz_j)\circ (id\bar z_j\wedge d\bar z_k\wedge d\bar z_l)\\
        &=-2\sqrt{2}id\bar z_k\wedge d\bar z_l
    \end{align*}
    Also observe that
    \begin{align*}
        *(dz_j\wedge d\bar z_j\wedge d\bar z_k\wedge d\bar z_l)&=*(-2idx_j\wedge dy_j\wedge d\bar z_k\wedge d\bar z_l)\\
        &=-2id\bar z_k\wedge d\bar z_l
    \end{align*}
    So we get 
    \begin{align*}
        c(dz_j\wedge\omega)(d\bar z_j\wedge d\bar z_k\wedge d\bar z_l)&=c(dz_j\wedge dx_k\wedge dy_k+dz_j\wedge dx_l\wedge dy_l)(d\bar z_j\wedge d\bar z_k\wedge d\bar z_l)\\
        &=-4\sqrt{2}i (d\bar z_k\wedge d\bar z_l)\\
        &=2\sqrt{2}*(dz_j\wedge d\bar z_j\wedge d\bar z_k\wedge d\bar z_l) 
    \end{align*}
    Therefore, the second formula follows. The last two formulae follow from remark \ref{remark Clifford action}.
\end{proof}
We prove another lemma concerning the Clifford action of the K\"ahler form $\omega.$ This will be instrumental in solving the two curvature equations \eqref{Curvature 6d +}, \eqref{Curvature 6d -}. 
\begin{lemma}\label{Clifford omega}
    We have the positive and negative spin bundles as the summands $S_+=\mathcal{L}_0\oplus(\Lambda^{0,2}\otimes\mathcal{L}_0)$ and $\tilde{S}_-=(\Lambda^{0,1}\otimes\mathcal{L}_1)\oplus (K_X^{-1}\otimes\mathcal{L}_1).$ $c(\omega)$ preserves the splitting of $S_+$ and $\tilde{S}_-$ and acts on the summands as follows:
    \begin{align*}
        c(\omega)=\begin{bmatrix}
            -3i\text{Id}&0\\
            0&i\text{Id}
        \end{bmatrix} \text{ on } S_+\text{ and }
        c(\omega)=\begin{bmatrix}
            -i\text{Id}&0\\
            0&3i\text{Id}
        \end{bmatrix} \text{ on } \tilde{S}_-
    \end{align*}
\end{lemma}
\begin{proof}
    We take local holomorphic coordinates $\{z_k=x_k+iy_k\}_{k=1,2,3}$ centered at a point $x\in X$ so that the Kähler metric is standard to second order at the point.\par
    \underline{$c(\omega)$ acting on $\Omega^0(X,\mathcal{L}_0):$} Take $j\in\{1,2,3\}$ and $\phi^0\in\Omega^0(X,\mathcal{L}_0)$.
    Using formula \eqref{Clifford one form} we compute: 
    \begin{align*}
        c(dx_j\wedge dy_j)\phi^0&=c(dx_j)\circ c(dy_j)\phi^0\\
        &=c(dx_j)(\frac{i}{\sqrt{2}}d\bar z_j\wedge\phi^0)\\
        &=-i\phi^0
    \end{align*}
    \begin{align*}
    \text{Hence, }c\big(\sum_j dx_j\wedge dy_j\big)\phi^0=-3i\phi^0
    \end{align*}\par
    \underline{$c(\omega)$ acting on $\Omega^{0,2}(X,\mathcal{L}_0):$} Take $\mu\in\Omega^{0,2}(X,\mathcal{L}_0)$ and say in local coordinates
    \begin{align*}
        \mu=\mu_{12} d\bar z_1\wedge d\bar z_2+\mu_{23} d\bar z_2\wedge d\bar z_3+\mu_{31} d\bar z_3\wedge d\bar z_1
    \end{align*}
    We compute
    \begin{align*}
        c(dx_1\wedge dy_1)\mu=i(\mu_{12} d\bar z_1\wedge d\bar z_2-\mu_{23} d\bar z_2\wedge d\bar z_3+\mu_{31} d\bar z_3\wedge d\bar z_1),\\
        c(dx_2\wedge dy_2)\mu=i(\mu_{12} d\bar z_1\wedge d\bar z_2+\mu_{23} d\bar z_2\wedge d\bar z_3-\mu_{31} d\bar z_3\wedge d\bar z_1),\\
        c(dx_3\wedge dy_3)\mu=i(-\mu_{12} d\bar z_1\wedge d\bar z_2+\mu_{23} d\bar z_2\wedge d\bar z_3+\mu_{31} d\bar z_3\wedge d\bar z_1)
    \end{align*}
    \begin{align*}
    \text{Thereafter, }   c\big(\sum_j dx_j\wedge dy_j\big)\mu=i\mu 
    \end{align*}\par 
    \underline{$c(\omega)$ acting on $\Omega^{0,1}(X,\mathcal{L}_0):$} Take $\mu\in\Omega^{0,1}(X,\mathcal{L}_0)$ and say in local coordinates
    \begin{align*}
        \mu=\mu_1 d\bar z_1+\mu_2 d\bar z_2+\mu_3 d\bar z_3
    \end{align*}
    We compute
    \begin{align*}
        c(dx_1\wedge dy_1)\mu=i(\mu_1 d\bar z_1-\mu_2 d\bar z_2-\mu_3 d\bar z_3),\\
        c(dx_2\wedge dy_2)\mu=i(-\mu_1 d\bar z_1+\mu_2 d\bar z_2-\mu_3 d\bar z_3),\\
        c(dx_3\wedge dy_3)\mu=i(-\mu_1 d\bar z_1-\mu_2 d\bar z_2+\mu_3 d\bar z_3)
    \end{align*}
    \begin{align*}
        \text{Hence, }c\big(\sum_j dx_j\wedge dy_j\big)\mu=-i\mu
    \end{align*}\par
    \underline{$c(\omega)$ acting on $\Omega^{0,3}(X,\mathcal{L}_0):$} For any \(j \in \{1, 2, 3\}\), we have  
\begin{align*}
    c(dx_j \wedge dy_j)(d\bar{z}_1 \wedge d\bar{z}_2 \wedge dz_3) = i (d\bar{z}_1 \wedge d\bar{z}_2 \wedge dz_3).
\end{align*}
Summing over \(j = 1, 2, 3\), we obtain  
\begin{align*}
    c\Big(\sum_j dx_j \wedge dy_j\Big)(d\bar{z}_1 \wedge d\bar{z}_2 \wedge dz_3) = 3i (d\bar{z}_1 \wedge d\bar{z}_2 \wedge dz_3). \hspace{2 ex}\tag*{\qedhere} 
\end{align*}
\end{proof}

\begin{corollary}\label{Clifford omega^2}
    For $S_+=\mathcal{L}_0\oplus(\Lambda^{0,2}\otimes\mathcal{L}_0)$ and $\tilde{S}_-=(\Lambda^{0,1}\otimes\mathcal{L}_1)\oplus (K_X^{-1}\otimes\mathcal{L}_1).$ $c(\omega^2)$ preserves the splitting of $S_+$ and $\tilde{S}_-$ and acts on the summands as follows:
    \begin{align*}
        c(\omega^2)=\begin{bmatrix}
            -6\text{Id}&0\\
            0&2\text{Id}
        \end{bmatrix} \text{ on } S_+\text{ and }
        c(\omega^2)=\begin{bmatrix}
            2\text{Id}&0\\
            0&-6\text{Id}
        \end{bmatrix} \text{ on } \tilde{S}_-
    \end{align*}
\end{corollary}
\begin{proof}
    Using lemma $4.5$ in \cite{JP} observe that 
    \begin{align*}
        c(\omega^2)=-2ic(\omega) \text{ on } S_+\\
        c(\omega^2)=2ic(\omega) \text{ on } S_-
    \end{align*}
    The corollary follows.
\end{proof}

\subsection{Solution on closed K\"ahler $3$-folds}\label{Proof 6d}
We construct a solution to the $6$-dimensional Seiberg--Witten equations on a closed K\"ahler threefold \((X, \omega)\). For this construction, we require two additional lemmas from K\"ahler geometry.
\begin{lemma}\label{Lemma 1}
 For a smooth complex valued function $h$ on a Kähler $3$-fold $(X,\omega),$ we have 
\begin{equation*}
    *(\partial\bar\partial h\wedge\omega)=-\partial\bar\partial h+\frac{i}{2}\Delta(h)\omega
\end{equation*}
\end{lemma}
\begin{proof}
We take local holomorphic coordinates $\{z_k=x_k +i y_k\}_{k=1,2,3}$ centered at a point $x\in X$ so that the Kähler metric is standard to second order at the point. All calculations in the following are performed at the point $x$. We compute
\begin{align*}
    *\big(\partial\bar\partial h\wedge\omega\big)&=*\Big(\sum\limits_{j,k=1}^{3} \frac{\partial^2h}{\partial z_j\partial\bar{z}_k}d z_j\wedge d\bar{z}_k\wedge\omega\Big)\\
    &=*\Big(\sum\limits_{j\neq k}^{} \frac{\partial^2h}{\partial z_j\partial\bar{z}_k}d z_j\wedge d\bar{z}_k\wedge\omega+\sum_{k=1}^3 \frac{\partial^2h}{\partial z_k\partial\bar{z}_k}d z_k\wedge d\bar{z}_k\wedge\omega\Big)
\end{align*}
Notice that for $j\neq k,$
\begin{align*}
    *(d z_j\wedge d\bar{z}_k\wedge\omega)=-(d z_j\wedge d\bar{z}_k)
\end{align*}
and 
\begin{align*}
    *\big(d z_k\wedge d\bar{z}_k\wedge\omega\big)&=*\big(-2i d x_k\wedge d y_k\wedge (\sum\limits_{m=1}^3d x_m\wedge d y_m)\big)\\
    &=-2i\big(\sum\limits_{m\neq k}d x_m\wedge d y_m\big)\\
    &= -2i(\omega-d x_k\wedge d y_k)\\
    &=-d z_k\wedge d\bar{z}_k-2i\omega
\end{align*}
Hence we get
\begin{align*}
    &*\Big(\sum\limits_{j\neq k}^{} \frac{\partial^2h}{\partial z_j\partial\bar{z}_k}d z_j\wedge d\bar{z}_k\wedge\omega+\sum_{k=1}^3 \frac{\partial^2h}{\partial z_k\partial\bar{z}_k}d z_k\wedge d\bar{z}_k\wedge\omega\Big)\\
    &=-\sum\limits_{j\neq k} \frac{\partial^2h}{\partial z_j\partial\bar{z}_k}d z_j\wedge d\bar{z}_k-\sum_{k=1}^3 \frac{\partial^2h}{\partial z_k\partial\bar{z}_k}d z_k\wedge d\bar{z}_k-2i\Big(\sum_{k=1}^3 \frac{\partial^2h}{\partial z_k\partial\bar{z}_k}\Big)\omega\\
    &=-\Big(\sum\limits_{j,k=1}^{3} \frac{\partial^2h}{\partial z_j\partial\bar{z}_k}d z_j\wedge d\bar{z}_k\Big)-2i \Big(\sum\limits_{k=1}^{3}\frac{1}{2}(\frac{\partial}{\partial x_k}-i\frac{\partial}{\partial y_k})\frac{1}{2}(\frac{\partial}{\partial x_k}+i\frac{\partial}{\partial y_k})h\Big)\omega\\
    &=-\partial\bar\partial h-\frac{2i}{4}\Big(\sum\limits_{k=1}^{3}\big(\frac{\partial^2 h}{\partial x_k^2}+\frac{\partial^2 h}{\partial y_k^2}\big)\Big)\omega\\
    &=-\partial\bar\partial h+\frac{i}{2}\Delta(h)\omega\tag*{\qedhere} 
\end{align*}
\end{proof}

\begin{lemma}\label{lemma 2}
For a smooth complex valued function $h$ on a Kähler $n$-fold $(X,\omega)$, we have 
\begin{equation*}
    \langle\partial\bar\partial h,\omega\rangle=\frac{i}{2}\Delta(h)
\end{equation*}
\end{lemma}
\begin{proof}
We take local holomorphic coordinates $\{z_k=x_k +i y_k\}_{k=1,\dots,n}$ centered at a point $x\in X$ so that the Kähler metric is standard to second order at the point. All calculations in the following are performed at the point $x$. 
\begin{align*}
    \langle\partial\bar\partial h,\omega\rangle&=*\big(\partial\bar\partial h\wedge(*\omega)\big) \\
    &=*\Big(\sum\limits_{j,k=1}^{n} \frac{\partial^2h}{\partial z_j\partial\bar{z}_k}d z_j\wedge d\bar{z}_k\wedge (*\omega)\Big)
\end{align*}
Notice that for $j\neq k, d z_j\wedge d\bar{z}_k\wedge (*\omega)=0.$ Hence we get
\begin{align*}
    &*\Big(\sum\limits_{j,k=1}^{n} \frac{\partial^2h}{\partial z_j\partial\bar{z}_k}d z_j\wedge d\bar{z}_k\wedge (*\omega)\Big)\\
    &=*\Big(\sum\limits_{k=1}^{n} \frac{\partial^2h}{\partial z_k\partial\bar{z}_k}d z_k\wedge d\bar{z}_k\wedge (*\omega)\Big)\\
    &=-2i*\Big(\sum\limits_{k=1}^{n} \frac{\partial^2h}{\partial z_k\partial\bar{z}_k}d x_k\wedge d{y}_k\wedge (*\omega)\Big)\\
    &=-2i*\Big(\sum\limits_{k=1}^{n}\frac{1}{2}\big(\frac{\partial}{\partial x_k}-i\frac{\partial}{\partial y_k}\big)\frac{1}{2}\big(\frac{\partial}{\partial x_k}+i\frac{\partial}{\partial y_k}\big)h\hspace{0.5 ex} d\text{vol}\Big)\\
    &=\frac{-2i}{4}\Big(\sum\limits_{k=1}^{n}\big(\frac{\partial^2 h}{\partial x_k^2}+\frac{\partial^2 h}{\partial y_k^2}\big)\Big)\\
    &=\frac{i}{2}\Delta (h)\tag*{\qedhere} 
\end{align*}
\end{proof}
\begin{proof}[Proof of theorem \ref{Theorem 6d}]
Let $A_0$ be the Chern connection on $\mathcal{L}_0$ with respect to its hermitian metric, say $h_0$. Take a nontrivial holomorphic section $\varphi$ of $\mathcal{L}_0$ (this assumes dim $H^0(X,\mathcal{L}_0)>0$), i.e., $\bar\partial_{A_0}\varphi=0.$ It follows from lemma \ref{Clifford action formulae} that for any $f\in C^\infty(X,\mathbb{C}),\phi=e^{3i(1+i)f}\varphi,A=-A_{K_X}+2A_0$ and $\beta=(\bar\partial f+\partial \bar f)\wedge\omega$ solves the first Dirac equation \eqref{Dirac 8d}:
\begin{align*}
    \big(D_A+c(*\beta)\big)\phi=\sqrt{2}\bar\partial_{A_0}\phi+ic(\bar\partial f \wedge\omega)\phi=0
\end{align*}
Since $\phi\in\Omega^0(X,\mathcal{L}_0),$ using lemma \ref{Clifford omega} we identify the quadratic term $q(\phi):$
\begin{align*}
    q(\phi)=\frac{i}{4}|\phi|^2\omega
\end{align*}
Hence, the first curvature equation \ref{Curvature 6d +} simplifies to
\begin{align}
    (F_A+2id^*\beta)^{1,1}&=\frac{i}{4}|\phi|^2\omega\label{1,1,+},\\
    F_A^{0,2}&=0\label{0,2,+}
\end{align}
Equation \ref{0,2,+} is automatically solved: $F_A^{0,2}=2F_{A_0}^{0,2}=0,$ since $A_0$ is the Chern connection on $\mathcal{L}_0.$ Using lemma \ref{Lemma 1} we compute 
\begin{align*}
    d^*\beta=-2i*\big(\partial\bar\partial(\text{Re}f)\wedge\omega\big)=2i\partial\bar\partial(\text{Re}f)+\Delta(\text{Re}f)\omega
\end{align*}
Equation \eqref{1,1,+} reads
\begin{align}\label{Curvature draft}
    F_A-4\partial\bar\partial\text{Re}f+2i\Delta(\text{Re}f)\omega=\frac{i}{4}e^{-4\text{Re}f}|\varphi|^2\omega
\end{align}
The initial hermitian metric on $\mathcal{L}_0$ is $h_0,$ we make a conformal change to this metric by $e^{2\lambda}$ for a smooth function $\lambda:X\rightarrow\mathbb{R}.$ Call this new metric $h_0'=e^{2\lambda} h_0.$ Notice since the change in metric is conformal $\varphi$ is still a holomorphic section of $(\mathcal{L}_0,h_0'),$ i.e., if $A_0'$ is the Chern connection on $\mathcal{L}_0$ with respect to $h_0',\bar\partial_{A_0'}\varphi=0.$ Call the new induced connection on $L(S)$ by $A'.$ Then with respect to $A'$ and $h_0'$ equation \eqref{Curvature draft} reads
\begin{align}\label{Curvature draft 2}
    F_{A'}-4\partial\bar\partial\text{Re}f+2i\Delta(\text{Re}f)\omega=\frac{i}{4}e^{-4\text{Re}f}|\varphi|_{h_0'}^2\omega
\end{align}
The assumption $c_1(K_X^{-1}\otimes \mathcal{L}_0^2)=a_0[\omega]$ with $a_0<0$ gives us $F_A=4\partial\bar\partial f_0-2\pi ia_0\omega$ for some real valued smooth function $f_0$. Substituting this in equation \eqref{Curvature draft 2} we get
\begin{align*}
    4\partial\bar\partial{(f_0-\lambda-\text{Re}f)}+\big(2i\Delta(\text{Re}f)-\frac{i}{4}e^{-4\text{Re}f}e^{2\lambda}|\varphi|_{h_0}^2-2\pi i a_0\big)\omega=0
\end{align*}
which breaks into the following two equations:
\begin{align}
    \lambda+\text{Re}f=f_0\label{Curvature break 1} \text{  and}\\
    2\Delta(-\text{Re}f)+\frac{1}{4}e^{2\lambda-4\text{Re}f}|\varphi|^2=-2\pi a_0\label{Curvature break 2}
\end{align}
Substituting $\lambda=f_0-\text{Re}f$ from \eqref{Curvature break 1} into the equation \eqref{Curvature break 2} we get
\begin{align}
    2\Delta(-\text{Re}f)+\frac{e^{2f_0}|\varphi|^2}{4}e^{-6\text{Re}f}=-2\pi a_0\label{Curvature break 3}
\end{align}
This type of pde is known as Kazdan--Warner equation in the literature and has a unique solution for Re$f$ since $-2\pi a_0>0$ and $\varphi$ being a non-trivial holomorphic section has isolated zeroes \cite{KW}. Once we have Re$f$ from \eqref{Curvature break 3}, we solve for $\lambda$ from equation \eqref{Curvature break 1}. Finally we make a $\mathbb{C}^*$ gauge transformation to produce a solution with respect to the original metric $h_0$ on $\mathcal{L}_0.$\par
We turn our attention to the second pair of equations \eqref{Dirac 6d -},\eqref{Curvature 6d -}. Let's denote the hermitian metric on $\mathcal{L}_1$ by $h_1$ and the corresponding Chern connection by $B_0$. We pick a non-trivial section $\xi\in\Omega^{0,3}(X,\mathcal{L}_1)\cong \Omega^0(X,K_X^{-1}\otimes \mathcal{L}_1)$ such that $\bar\partial_{B_0}^*\xi=0.$ This is equivalent to choosing a non-trivial holomorphic section $\bar\xi$ of $\bar{\mathcal{L}}_1\cong \mathcal{L}_1^*,$ so we are indeed using the assumption that dim $H^0(X,K_X\otimes\mathcal{L}_1^{-1})> 0.$ Define $\psi:=e^{-2i\bar f}\xi$. A small calculation gives us
\begin{align*}
    \bar\partial^*_{B_0}\psi=\bar\partial^*_{B_0}(e^{-2i\bar f}\xi)=-2*(\partial\bar f\wedge\psi)    
\end{align*}
Using lemma \ref{Clifford action formulae} observe that $\psi,B_0,\beta=(\bar\partial f+\partial\bar f)\wedge\omega$ solves the Dirac equation \eqref{Dirac 6d -}:
\begin{align*}
    \big(D_B+c(\beta)\big)\psi=\sqrt{2}\bar\partial_{B_0}^*\psi+c(\partial\bar f\wedge\omega)\psi=0
\end{align*}
What remains to solve is the curvature equation \eqref{Curvature 6d -}. We use corollary \ref{Clifford omega^2} to see that for $\psi\in\Omega^{0,3}(X,\mathcal{L}_1)$, the quadratic term $q(\psi)$ can be identified as
\begin{align*}
    q(\psi)=-\frac{1}{8}|\psi|^2\omega^2
\end{align*}the curvature
equation \eqref{Curvature 6d -} breaks down into the following two equations:
\begin{align*}
    (-{i}*F_B+2d\beta)^{2,2}&=-\frac{1}{8}|\psi|^2\omega^2\\
    *F_B^{0,2}&=0
\end{align*}
Taking Hodge-star and multiplying by $i$ on both sides of these two equations we get
\begin{align}
    (F_B+2i*d\beta)^{1,1}&=-\frac{i}{4}|\psi|^2\omega\label{Curvature 2 break 1}\\
    F_B^{0,2}&=0\label{Curvature 2 break 2}
\end{align}
Equation \eqref{Curvature 2 break 2} is automatically solved: $F_B^{0,2}=2F_{B_0}^{0,2}=0$ since $B_0$ is the Chern connection on $\mathcal{L}_1.$
Using lemma \ref{Lemma 1} we calculate
\begin{align*}
    *d\beta=*\big(2i\partial\bar\partial(\text{Im}f)\wedge\omega\big)=-2i\partial\bar\partial(\text{Im}f)-\Delta(\text{Im}f)\omega
\end{align*}
Hence equation \eqref{Curvature 2 break 1} reads
\begin{align*}
    F_B+4\partial\bar\partial(\text{Im}f)-2i\Delta(\text{Im}f)\omega=-\frac{i}{4}e^{-4\text{Im}f}|\xi|^2\omega
\end{align*}
We perturb the initial metric $h_1$ on ${\mathcal{L}}_1$ by $e^{2\tilde\lambda},\tilde\lambda\in C^\infty(X,\mathbb{R}).$ The new metric being $\tilde{h}_1:=e^{2\tilde{\lambda}} h_1.$ Notice this conformal change in the metric doesn't change the holomorphic structure of $\mathcal{L}_1$ and induces a conformal change by $e^{-2\tilde\lambda}$ on the corresponding hermitian metric on $\mathcal{L}_1^{-1}.$ So, $\bar\xi$ remains a holomorphic section of $K_X-\mathcal{L}_1,$ and now when we go back to the corresponding anti-holomorphic section on $\mathcal{L}_1-K_X,$ the norm changes by $e^{-2\tilde\lambda}.$ 
Let's call the Chern connection on ${\mathcal{L}}_1$ with respect to the new metric ${\tilde{h}}_1$ by $B_0^{'}$ and the corresponding connection on $L(\tilde{S})$ by $B{'}.$ We write the equation with respect to the Chern connection $B'$ and the metric $\tilde{h}_1$ on $\mathcal{L}_1:$
\begin{align*}
    F_{B'}+ 4\partial\bar\partial(\text{Im}f)-2i\Delta(\text{Im}f)\omega=-\frac{i}{4}e^{-4\text{Im}f}|\xi|_{\tilde{h}_1}^2\omega
\end{align*}
Since $F_{B'}=F_B-4\partial\bar\partial\tilde{\lambda},$ the last equation reads
\begin{align}\label{curvature 1000}
    F_B+4\partial\bar\partial(\text{Im}f-2\tilde\lambda)-2i\Delta(\text{Im}f)\omega=-\frac{i}{4}e^{-4\text{Im}f-\tilde\lambda}|\xi|^2\omega
\end{align}
The assumption $c_1(K_X^{-1}\otimes \mathcal{L}_1^2)=a_1[\omega]$ (with $a_1>0$) gives us $F_B=4\partial\bar\partial f_1-2\pi ia_1\omega$ for a real valued smooth function $f_1$. Hence equation \eqref{curvature 1000} reads
\begin{align}\label{Curvature 1001}
    4\partial\bar\partial(f_1+\text{Im}f-\tilde\lambda)+i\big(2\Delta(-\text{Im}f)+\frac{1}{4}e^{-4\text{Im}f-2\tilde\lambda}|\xi|^2-2\pi a_1\big)\omega=0
\end{align}
Equation \eqref{Curvature 1001} breaks into two parts:
\begin{align}
    \tilde{\lambda}=f_1+\text{Im}f\label{Curvature 1002}\\
    2\Delta(-\text{Im}f_1)+\frac{1}{4}e^{-4\text{Im}f-2\tilde{\lambda}}|\xi|^2=2\pi a_1\label{Curvature 1003}
\end{align}
Substituting $\tilde\lambda$ from equation\eqref{Curvature 1002} into equation \eqref{Curvature 1003}, we get 
\begin{align}\label{Curvature 1004}
    2\Delta(-\text{Im}f)+\frac{e^{-2f_1}|\xi|^2}{4}e^{-6\text{Im}f}=2\pi a_1
\end{align}
$|\xi|^2$ has isolated zeroes since $\bar\xi$ is a holomorphic section of $K_X\otimes\mathcal{L}_1^{-1}$ and $2\pi a_1>0.$ Thereafter, equation \eqref{Curvature 1004} has a unique solution for Im$f$ \cite{KW}. Once we know Im$f$, we solve for $\tilde\lambda$ from equation \eqref{Curvature 1002}. In the end, we make a gauge transformation to produce a solution with respect to the original metric $h_1$ on $\mathcal{L}_1.$ 
\end{proof}
\subsection{Relationship with vortices}
\label{vortices}
From the proof above notice that there exist hermitian metrics on the two holomorphic line bundles on $\mathcal{L}_0$ and $\mathcal{L}_1$ such that for each non-trivial pair of $\varphi\in H^0(X,\mathcal{L}_0),\bar\xi\in H^0(X,K_X\otimes\mathcal{L}^{-1}_1),$ there is a unique way to solve for $f:X\rightarrow\mathbb{C}$ and the two unitary connections $A_0,B_0$ on $\mathcal{L}_0$ and $\mathcal{L}_1$ such that  
\begin{align*}
    \big(\phi=e^{-2f}\varphi,\psi=e^{-2i\bar f}\xi, A=(-A_{K_X}+2A_0),B=(-A_{K_X}+2B_0),\beta=(\bar\partial f+\partial\bar f)\wedge\omega\big)
\end{align*}
solves the $6$-dimensional Seiberg--Witten equations \eqref{Dirac 6d +},\eqref{Curvature 6d +},\eqref{Dirac 6d -},\eqref{Curvature 6d -}.\par 
We want to see what happens to this solution if we start with $a\varphi,b\xi$ for two non-zero constants $a,b\in\mathbb{C}^*.$ Define 
\begin{align*}
     f_{a,b}:=f+\frac{1}{2}\text{ln}|a|+\frac{i}{2}\text{ln}|b|
\end{align*}
If we would have started with $a\varphi$ and $b\bar\xi$ instead of $\varphi$ and $\bar\xi$, the above construction would give us a new solution to the  equations:
\begin{align*}
    &\big(e^{-2f_{a,b}}a\varphi,e^{-2i\bar f_{a,b}}b\xi,(-A_{K_X}+2A_0),(-A_{K_X}+2B_0),(\bar\partial f_{a,b}+\partial\bar f_{a,b})\wedge\omega\big)\\
    &=\Big(\frac{a}{|a|}e^{-i\text{ln}|b|}e^{-2f}\varphi,\frac{b}{|b|}e^{-i\text{ln}|a|}e^{-2i\bar f}\xi,(-2A_{K_X}+4A_0),(-2A_{K_X}+4B_0),(\bar\partial f+\partial\bar f)\wedge\omega\Big)
\end{align*}
Which is gauge-equivalent to the original solution
\begin{align*}
    \big(e^{-2f}\varphi,e^{-2i\bar f}\xi, (-A_{K_X}+2A_0),(-A_{K_X}+2B_0),\beta=(\bar\partial f+\partial\bar f)\wedge\omega\big)
\end{align*}
let's say that for $\mathcal{L}_0$ and $\mathcal{L}_1$, the four conditions mentioned in theorem \ref{Theorem 6d} are satisfied. Observe that the construction of a solution in \S\ref{Proof 6d} relies solely on the holomorphic structures $\mathcal{L}_0$ and $\mathcal{L}_1$ and of non-trivial holomorphic sections of \(\mathcal{L}_0\) and \(K_\Sigma\otimes\mathcal{L}_1^{-1}\). For a fixed choice of holomorphic structures on $\mathcal{L}_0,\mathcal{L}_1$ and two pairs of non-trivial sections which are equivalent modulo $\mathbb{C}^*\times\mathbb{C}^*$ action, we get two solutions of the Seiberg--Witten equations, which are $S^1\times S^1$ gauge equivalent. Thereafter, modulo gauge transformations, the space of solutions described in theorem \ref{Theorem 6d} can be identified with $\mathcal{M}_{\text{vortex}}\big(c_1(\mathcal{L}_0)\big)\times \mathcal{M}_{\text{vortex}}\big(c_1(K_X\otimes\mathcal{L}_1^{-1})\big).$

\subsection{Failure of compactness}\label{Compactness}
We give an explicit example where the compactness of the moduli space does indeed fail. In dimension $6$ this phenomenon is directly related to the existence of certain harmonic $3$ forms. Let's say $(A,\phi,B,\psi,\beta)$ solve the $6$-dimensional Seiberg--Witten equations \eqref{Dirac 6d +},\eqref{Curvature 6d +},\eqref{Dirac 6d -},\eqref{Curvature 6d -}. Now suppose that there is a harmonic $3$-form $\theta$ such that $c(*\theta)\phi=0$ and $c(\theta)\psi=0.$ Then for any $t\in\mathbb{R},(A,\phi,B,\psi,\beta+t\theta)$ is again a solution.\par
We look at our solutions of the Seiberg--Witten equations on closed K\"ahler three-folds. Following the construction in \S\ref{Proof 6d}, the two spinors $\phi\in\Omega^0(X,\mathcal{L}_0),\psi\in\Omega^{0,3}(X,\mathcal{L}_1).$ Suppose there exists a $\theta^{1,2}\in\Omega^{1,2}(X,\mathbb{C})$ such that $\theta^{1,2}$ is harmonic and $\theta^{1,2}\wedge\omega=0.$ Define $\theta=\theta^{1,2}+\overline{\theta^{1,2}}.$ Lemma \ref{lemma Lefschetz} and lemma \ref{Clifford action formulae} tell us
\begin{align*}
   c(*\theta)\phi=c(-i\overline{\theta^{1,2}})\phi=0\\
   c(\theta)\psi=c(\theta^{1,2})\psi=0 
\end{align*}
A harmonic $\theta^{1,2}\in\Omega^{1,2}$ such that $\theta^{1,2}\wedge\omega=0$ is an element of the primitive cohomology group $H^{1,2}_p$. Hence, this phenomenon can occur whenever $\phi\in\Omega^0(X,\mathcal{L}_0),\psi\in\Omega^{0,3}(X,\mathcal{L}_1)$ and $H^{1,2}_p\neq \{0\}.$ In example \ref{example 1}, where $X$ is the product of three Riemann surfaces of the same genus, we compute using K\"unneth: dim $H^{1,2}(X)=3$ and notice that $H^{1,2}(X)=H^{1,2}_p(X).$ Hence, $H^{1,2}_p(X)\neq\{0\}.$

\subsection{Taubes' limit}\label{Tabues' limit}
On a closed symplectic four-manifold $(M,\omega),$ with $b_+>1,\omega$ being a symplectic form on $M$, choosing a compatible almost complex structure $J$ also gives us a Riemannian metric. Moreover we have the canonical spinor bundle: $S^{\text{can}}=\Lambda^0\oplus\Lambda^{0,1}\oplus\Lambda^{0,2}$ and any other bundle of positive spinors is obtained by twisting $S^{\text{can}}$ by a complex line bundle. Hence the Seiberg--Witten invariant is a map: SW: $H^2(M,\mathbb{Z})\rightarrow\mathbb{Z}.$\par
Taubes defines another invariant  $\text{Gr}:H^2(M,\mathbb{Z})\rightarrow\mathbb{Z}$ called Gromov invariant \cite{Taubes5}. This is a Gromov--Witten type invariant, i.e., given an element of $H^2(M,\mathbb{Z})$ say $e,$ roughly speaking it gives a ``count" of pseudo-holomorphic curves whose homology class is $PD(e)$: Poincar\'e dual of $e.$ In a remarkable \textit{tour de force} Taubes also proves this two invariants  are same: ``SW$=$Gr" \cite{Taubes4}.\par
To prove one direction of this equivalence (``SW$\Rightarrow$ Gr") \cite{Taubes2}, Taubes studies a perturbed version of the Seiberg--Witten equations:
\begin{align}
    D_{A_r}\phi_r=0\label{Taubes Dirac}\\
    F_{A_r}^++F_{A_{K_M}}^++ir\omega=q(\phi_r)\label{Taubes curvature}
\end{align}
$r>>0$ and $A_{K_M}$ is the Chern connection on the canonical bundle $K_M$ with respect to the induced hermitian metric from the symplectic metric. Now say for the spin bundle $S^\text{can}\otimes\mathcal{L}$ ($\mathcal{L}$ is a complex line bundle) $(A_r,\phi_r)$ solves the equations \eqref{Taubes Dirac},\eqref{Taubes curvature}. We write
\begin{align*}
    \phi_r=\sqrt{r}(a_r+b_r)\in(\Omega^0\oplus\Omega^{0,2})\otimes\mathcal{L}.
\end{align*}
Taubes proves that as $r\rightarrow\infty, a_r^{-1}(0)$ converges to a pseudo-holomorphic curve with homology class $PD\big(c_1(L)\big)$ if not empty.\par
A similar story might exist for the the $6$-dimensional Seiberg--Witten equations. We take the base manifold to be a closed symplectic $3$-manifold $(X,\omega)$. Choose a compatible almost complex structure $J$. For two spinor bundles $S=S^{\text{can}}\otimes\mathcal{L}_0, \tilde{S}=S^{\text{can}}\otimes K_X\otimes\mathcal{L}_1$ ($\mathcal{L}_0,\mathcal{L}_1$ are line bundles) we define the perturbed Seiberg--Witten equations
\begin{align}
    \big(D_{A_r}+c(*\beta_r)\big)\phi_r=0\label{Dirac 6d + Taubes}\\
    F_{A_r}+F_{A_{K_X}}+2id^*\beta_r+ir\omega=q(\phi_r)\label{Curvature 6d + Taubes}\\
    \big(D_{B_r}+c(\beta_r)\big)\psi_r=0\label{Dirac 6d - Taubes}\\
    -{i}*(F_{B_r}-F_{A_{K_X}})+2d\beta_r-r\omega^2=q(\psi_r)\label{Curvature 6d - Taubes}
\end{align}
where $r>0$ and $A_{K_X}$ is the Chern connection on the canonical bundle $K_X$ with respect to the induced hermitian metric from the symplectic metric. Suppose $(A_r,\phi_r,B_r,\psi_r,\beta_r)$ is a solution to the  equations \eqref{Dirac 6d + Taubes},\eqref{Curvature 6d + Taubes},\eqref{Dirac 6d - Taubes},\eqref{Curvature 6d - Taubes}. We write
\begin{align*}
    \phi_r=\sqrt{r}(a_r+b_r)\in(\Omega^0\oplus\Omega^{0,2})\otimes\mathcal{L}_0\\
    \psi_r=\sqrt{r}(c_r+d_r)\in(\Omega^{0,1}\oplus\Omega^{0,3})\otimes (K_X\otimes\mathcal{L}_1)
\end{align*}
Notice $a_r\in \Omega^0(X,\mathcal{L}_0),d_r\in\Omega^0(X,\mathcal{L}_1)$ and if not empty $a^{-1}_r(0)$  and $d^{-1}_r(0)$ intersect along a surface with homology class $PD\big(c_1(\mathcal{L}_0)\smile c_1(\mathcal{L}_1)\big),\smile$ denoting the cup product between two elements of $H^2(X,\mathbb{Z}).$ One can hope to take the \textit{Taubes' limit} here, i.e., send $r\rightarrow\infty$ and see if the surface $a^{-1}_r(0)\cap d^{-1}_r(0)$ converges to a pseudo-holomorphic curve with homology class $PD\big(c_1(\mathcal{L}_0)\smile c_1(\mathcal{L}_1)\big)$.\par 
If $J$ is integrable, i.e., the manifold is K\"ahler, one can follow the proof of theorem \ref{Theorem 6d} and see that the same proof gives us the existence of a non-trivial solution to the  equations above for large enough $r$. Concretely speaking, there exists an $R>0$ (depending on $\mathcal L_0,\mathcal L_1)$ such that for all $r>R,$ there exists a non-trivial solutions to the equations \eqref{Dirac 6d + Taubes},\eqref{Curvature 6d + Taubes},\eqref{Dirac 6d - Taubes},\eqref{Curvature 6d - Taubes} under the following assumptions:
\begin{align*}
    &1.\hspace{1 ex} \text{dim }H^0(X,\mathcal{L}_0)>0\\
    &2. \hspace{1 ex}c_1(\mathcal{L}_0)=\tilde{a}_0[\omega]\hspace{1 ex}(\tilde{a}_0\in\mathbb{R})\\
    &3. \hspace{1 ex} \text{dim }H^0(X,\mathcal{L}_1^{-1})>0\\
    &4. \hspace{1 ex}c_1(\mathcal{L}_1)=\tilde{a}_1[\omega]\hspace{1 ex}(\tilde{a}_1\in\mathbb{R}).
\end{align*}
Moreover in that case, both $c_r,d_r$ would be zero. $a^{-1}_r(0)$  and $d^{-1}_r(0)$ would be divisors, being the zero set of holomorphic sections of $\mathcal{L}_0$ and $\mathcal{L}_1^{-1}$ respectively. Not only that, for all $r>R$, $a^{-1}_r(0)$  and $d^{-1}_r(0)$ would be the same pair of divisors independent of $r$ and hence generically they would intersect along a smooth complex curve.\par
Taubes also proves that for a compact symplectic 4-manifold with $b_+>1$ the Seiberg--Witten invariant for the canonical spin$^\mathbb{C}$-structure is always $1$ \cite{Taubes1}. Following our construction of the solution on a closed K\"ahler $3$-fold $X$ in \S\ref{Proof 6d}, if $b_1=b_3=0$ and the four necessary conditions are satisfied for $\mathcal{L}_0$ and $\mathcal{L}_1,$ the moduli space of the constructed solutions is $\mathbb{CP}(H^0(X,\mathcal{L}_0))\times \mathbb{CP}(H^0(X,K_X\otimes\mathcal{L}_1^{-1}))$. Moreover, if $\mathcal{L}_0=\underline{\mathbb{C}},$ the trivial line bundle and $\mathcal{L}_1=K_X,$ the canonical line bundle; then $\mathbb{CP}(H^0(X,\mathcal{L}_0))\times \mathbb{CP}(H^0(X,K_X\otimes\mathcal{L}_1^{-1}))$ becomes a singleton set.
\subsection{Solution on $\Sigma\times\mathbb{C}^2$}\label{Riemann surface times C2}
The product metric on $\Sigma\times\mathbb{C}^2$ induces spin bundles of opposite chirality for two holomorphic hermitian line bundles $\mathcal{L}_0,\mathcal{L}_1$ on $\Sigma$:
\begin{align*}
    S_+(\Sigma\times\mathbb{C}^2)=\Big(\pi_1^*(\mathcal{L}_0)\otimes \pi_2^*\big(S_+(\mathbb{C}^2)\big)\Big)\oplus \Big(\pi_1^*(K_\Sigma^{-1}\otimes\mathcal{L}_0)\otimes \pi_2^*\big(S_-(\mathbb{C}^2)\big)\Big)\\
    \tilde{S}_-(\Sigma\times\mathbb{C}^2)=\Big(\pi_1^*(\mathcal{L}_1)\otimes \pi_2^*\big(S_-(\mathbb{C}^2)\big)\Big)\oplus \Big(\pi_1^*(K_\Sigma^{-1}\otimes\mathcal{L}_1)\otimes \pi_2^*\big(S_+(\mathbb{C}^2)\big)\Big)
\end{align*}
The associated line bundles are
\begin{align*}
    L\big(S_+(\Sigma\times\mathbb{C}^2)\big)=\pi_1^*(K_\Sigma^{-1}\otimes \mathcal{L}_0^2)\\
    L\big(\tilde{S}_-(\Sigma\times\mathbb{C}^2)\big)=\pi_1^*(K_\Sigma^{-1}\otimes \mathcal{L}_0^2)
\end{align*}
The K\"ahler metric on $\Sigma$ determines a metric on $K_\Sigma$. We call the corresponding Chern connection with respect to this metric by $A_{K_\Sigma}.$ For any two unitary connections $A_0$ and $A_1$ respectively on $\mathcal{L}_0$ and $\mathcal{L}_1$, we get two unitary connections $A$ and $B$ on $L\big(S_+(\Sigma\times\mathbb{C}^2)\big)$ and $L\big(\tilde{S}_-(\Sigma\times\mathbb{C}^2)\big)$ via $A=\pi_1^*(-A_{K_\Sigma}+2A_0)$ and $B=\pi_1^*(-A_{K_\Sigma}+2A_1)$ (with abuse of notation). We call the K\"ahler form on $\Sigma\times\mathbb{C}^2$ by $\omega:=\pi_1^*(\omega_\Sigma)+\pi_2^*(\omega_{\mathbb{C}^2}).$
\begin{proof}[Proof of theorem \ref{Theorem 6d sigma}] The proof is very similar to the proof of theorem \ref{Theorem 6d}. We start with the Chern connection, say $A_0$ on $\mathcal{L}_0$ with respect to the given hermitian metric on it (say $h_0).$ Take a nontrivial holomorphic section $\varphi$ of $\mathcal{L}_0$ (this uses the assumption that dim $H^0(\Sigma,\mathcal{L}_0)>0$), i.e., $\bar\partial_{A_0}\varphi=0.$ It follows from lemma \ref{Clifford action formulae} that for any $f\in C^\infty(\Sigma,\mathbb{C}),\phi=\pi_1^*(e^{-2f}\varphi)\otimes \pi_2^*(1)\in\Gamma(\pi_1^*\mathcal{L}_0\otimes\pi_2^*\Lambda^0),\beta=\pi_1^*(\bar\partial f+\partial \bar f)\wedge\omega=\big(\bar\partial(\pi_1^* f)+\partial  (\overline {\pi_1^*f})\big)\wedge\omega$ and $A=\pi_1^*(-A_{K_\Sigma}+2A_0)$ solves the first Dirac equation \eqref{Dirac 6d + sigma}:
\begin{align*}
    \big(D_A+c(*\beta)\big)\phi=\sqrt{2}\pi_1^*\big(\bar\partial_{A_0}(e^{-2f}\varphi)\big)\otimes \pi_2^*(1)+ic\big(\bar\partial(\pi_1^*f) \wedge\omega)\big)\phi=0
\end{align*}
Using lemma \ref{Clifford omega} we identify the quadratic term $q(\phi):$
\begin{align*}
    q(\phi)=\frac{i}{4}|\phi|^2\omega
\end{align*}
The first curvature equation \ref{Curvature 6d + sigma} simplifies to
\begin{align}
    (F_A+2id^*\beta)^{1,1}+\pi_2^*(ir_0\omega_{\mathbb{C}^2})&=\frac{i}{4}|\phi|^2\omega\label{1,1,+ sigma},\\
    F_A^{0,2}&=0\label{0,2,+ sigma}
\end{align}
Equation \ref{0,2,+ sigma} is automatically solved: $F_A^{0,2}=2\pi_1^*(F_{A_0}^{0,2})=0,$ since $A_0$ is the Chern connection on $\mathcal{L}_0.$ Using lemma \ref{Lemma 1} and \ref{lemma 2} we compute 
\begin{align*}
    d^*\beta=-2i*\big(\partial\bar\partial(\pi_1^*\text{Re}f)\wedge\omega\big)=2i\partial\bar\partial(\pi_1^*\text{Re}f)+\Delta(\pi_1^*\text{Re}f)\omega=\pi_1^*(\Delta_\Sigma\text{Re}f)\pi_2^*(\omega_{\mathbb{C}^2})
\end{align*}
Equation \eqref{1,1,+ sigma} breaks into two parts:
\begin{align}
    2\Delta_\Sigma(-\text{Re}f)+\frac{1}{4}e^{-4\text{Re}f}|\varphi|^2=r_0\label{sigma + 1}\\
    2F_{A_0}-F_{K_\Sigma}=\frac{i}{4}e^{-4\text{Re}f}|\varphi|^2\omega_\Sigma\label{sigma +2}
\end{align}
The initial hermitian metric on $\mathcal{L}_0$ is $h_0,$ we make a conformal change to this metric by $e^\lambda$ for a smooth function $\lambda:\Sigma\rightarrow\mathbb{R}.$ Call this new metric $h_0'=e^\lambda h_0.$ Notice since the change in metric is conformal $\varphi$ is still a holomorphic section of $(\mathcal{L}_0,h_0'),$ i.e., if $A_0'$ is the Chern connection on $\mathcal{L}_0$ with respect to $h_0',\bar\partial_{A_0'}\varphi=0.$ With respect to $A_0'$ and $h_0'$ equations \eqref{sigma + 1} and \eqref{sigma +2} read
\begin{align}
    2\Delta_\Sigma(-\text{Re}f)+\frac{1}{4}e^{-4\text{Re}f}|\varphi|^2_{h_0'}=r_0\label{sigma + 100}\\
    2F_{A_0'}-A_{K_\Sigma}=\frac{i}{4}e^{-4\text{Re}f}|\varphi|^2_{h_0'}\omega_\Sigma\label{sigma +3}
\end{align}
Since $F_{A_0'}=F_{A_0}-\partial\bar\partial\lambda,$ equation \eqref{sigma +3} reads
\begin{align*}
    2F_{A_0}-2\partial\bar\partial\lambda-A_{K_\Sigma}=\frac{i}{4}e^{\lambda-4\text{Re}f}|\varphi|^2\omega_\Sigma
\end{align*}
Taking point wise inner product with $\omega_\Sigma$ on both sides yields
\begin{align}
    \langle 2F_{A_0}-F_{A_\Sigma},\omega_\Sigma\rangle-i\Delta_\Sigma\lambda=\frac{i}{4}e^{\lambda-4\text{Re}f}|\varphi|^2 
\end{align}
which is same as 
\begin{align}
    \Delta_\Sigma\lambda+\frac{1}{4}e^{\lambda-4\text{Re}f}|\varphi|^2 =-i\langle 2F_{A_0}-F_{A_\Sigma},\omega_\Sigma\rangle\label{sigma +4}
\end{align}
Multiplying both sides of the equation \eqref{sigma + 100} by $2$ and adding it to the equation \eqref{sigma +4} we get
\begin{align}\label{eq 54}
    \Delta_\Sigma(\lambda-4\text{Re}f)+\frac{3|\varphi|^2}{4}e^{(\lambda-4\text{Re}f)}=2r_0+i\langle F_{A_\Sigma}-2F_{A_0},\omega_\Sigma\rangle
\end{align}
$|\varphi|^2$ has isolated zeroes and since $r_0=\frac{2\pi\text{deg}(K_\Sigma-2\mathcal{L})}{\text{vol}(\Sigma)}>0,$ we have $\int_\Sigma (2r_0+i\langle F_{A_\Sigma}-2F_{A_0},\omega_\Sigma\rangle)\omega_\Sigma>0.$ Therefore, there exists a unique solution for $(\lambda-4\text{Re}f)$ from equation \eqref{eq 54} \cite{KW}. On the other hand, multiplying both sides of the equation \eqref{sigma + 100} by $-1$ and adding it to the equation \eqref{sigma +4} we get
\begin{align}\label{eq 89}
   \Delta_\Sigma(2\text{Re}f+\lambda)=i\langle F_{A_\Sigma}-2F_{A_0},\omega_\Sigma\rangle-r_0 
\end{align}
Since $r_0=\frac{2\pi\text{deg}(K_\Sigma-2\mathcal{L}_0)}{\text{vol}(\Sigma)}, \int_\Sigma (i\langle F_{A_\Sigma}-2F_{A_0},\omega_\Sigma\rangle-r_0 )\omega_\Sigma=0.$ Hence there exists a unique solution for $(2\text{Re}f+\lambda)$ from equation \eqref{eq 89}. As we have solved both functions $(\lambda-4\text{Re}f)$ and $2\text{Re}f+\lambda,$ we can determine \(\text{Re}f\) and \(\lambda\) individually. Finally, we perform a \(\mathbb{C}^*\) gauge transformation to obtain a solution with respect to the original metric \(h_0\).

Now we solve the second pair of equations  \eqref{Dirac 6d - sigma},\eqref{Curvature 6d - sigma}. Let's denote the hermitian metric on $\mathcal{L}_1$ by $h_1$ and the corresponding Chern connection by $B_0$. We pick a non-trivial section $\xi\in\Omega^{0,3}(\Sigma,\mathcal{L}_1)\cong \Omega^0(\Sigma,K_\Sigma^{-1}\otimes \mathcal{L}_1)$ such that $\bar\partial_{B_0}^*\xi=0.$ This is equivalent to choosing a non-trivial holomorphic section $\bar\xi$ of $\bar{\mathcal{L}}_1\cong \mathcal{L}_1^*,$ so we are indeed using the assumption that dim $H^0(\Sigma,K_\Sigma\otimes\mathcal{L}_1^{-1})> 0.$ Define $\psi:=\pi_1^*(e^{-2i\bar f}\xi)\otimes\pi_2^*(1)\in \Gamma\big(\pi_1^*(K_\Sigma^{-1}\otimes\mathcal{L}_1)\otimes\pi_2^*\Lambda^0\big)$. A small calculation gives us
\begin{align*}
    D_B(\psi)=\sqrt{2}\pi_1^*\big(\bar\partial^*_{B_0}(e^{-2i\bar f}\xi)\big)\otimes\pi_2^*(1)=-2\sqrt{2}\pi_1^*\Big(*_\Sigma\big(\partial\bar f\wedge (e^{-2i\bar f}\xi)\big)\Big)\otimes\pi_2^*(1)
\end{align*}
Observe that $\psi,B,\beta=\pi_1^*(\bar\partial f+\partial \bar f)\wedge\omega$ solves the Dirac equation \eqref{Dirac 6d - sigma}:
\begin{align*}
    \big(D_B+c(\beta)\big)\psi&=D_B(\psi)+c(\pi_1^*(\partial\bar f)\wedge\omega_{\mathbb{C}^2})\psi\\
    &=D_B(\psi)+\pi_1^*\big(c(\partial\bar f)(e^{-2i\bar f}\xi)\big)\otimes \pi_2^*\big(c(\omega_{\mathbb{C}^2})1\big)\\
    &=0
\end{align*}
What remains to solve is the curvature equation \eqref{Curvature 6d - sigma}. We use corollary \ref{Clifford omega^2} to see that the quadratic term $q(\psi)$ can be identified as
\begin{align*}
    q(\psi)=-\frac{1}{8}|\psi|^2\omega^2
\end{align*}the curvature
equation \eqref{Curvature 6d - sigma} breaks down into the following two equations:
\begin{align*}
    \big(-{i}*F_B+2d\beta+i*\pi_2^*(ir_1\omega_{\mathbb{C}^2})\big)^{2,2}&=-\frac{1}{8}|\psi|^2\omega^2\\
    *F_B^{0,2}&=0
\end{align*}
Taking Hodge-star and multiplying by $i$ on both sides of these two equations we get
\begin{align}
    \big(F_B+2i*d\beta-\pi_2^*(ir_1\omega_{\mathbb{C}^2})\big)^{1,1}&=-\frac{i}{4}|\psi|^2\omega\label{Curvature 2 break 1 sigma}\\
    F_B^{0,2}&=0\label{Curvature 2 break 2 sigma}
\end{align}
Equation \eqref{Curvature 2 break 2 sigma} is automatically solved: $F_B^{0,2}=2\pi_1^*(F_{B_0}^{0,2})=0,$ since $B_0$ is the Chern connection on $\mathcal{L}_1.$
Using lemma \ref{Lemma 1} and \ref{lemma 2} we calculate
\begin{align*}
    *d\beta&=*\big(2i\partial\bar\partial(\pi_1^*(\text{Im}f))\wedge\omega\big)\\
    &=-2i\partial\bar\partial(\pi_1^*(\text{Im}f))-\Delta(\pi_1^*(\text{Im}f))\omega\\
    &=-\pi_1^*\big(\Delta_\Sigma (\text{Im}f)\big)\pi_2^*(\omega_{\mathbb{C}^2})
\end{align*}
Equation \eqref{Curvature 2 break 1 sigma} breaks into two equations:
\begin{align}
    2\Delta_\Sigma(-\text{Im}f)+\frac{1}{4}e^{-4\text{Im}f}|\xi|^2=r_1\label{curature 2 - 1sigma}\\
    2F_{B_0}-F_{K_\Sigma}=-\frac{i}{4}e^{-4\text{Im}f}|\xi|^2\omega_\Sigma\label{curature 2 - 2sigma}
\end{align}
We perturb the initial metric $h_1$ on ${\mathcal{L}}_1$ by $e^{\tilde\lambda},\tilde\lambda\in C^\infty(\Sigma,\mathbb{R}).$ The new metric being $\tilde{h}_1:=e^{\tilde{\lambda}} h_1.$ Notice this conformal change in the metric doesn't change the holomorphic structure of $\mathcal{L}_1$ and induces a conformal change by $e^{-\tilde\lambda}$ on the corresponding hermitian metric on $\mathcal{L}_1^{-1}.$ So, $\bar\xi$ remains a holomorphic section of $K_\Sigma-\mathcal{L}_1,$ and now when we go back to the corresponding anti-holomorphic section on $\mathcal{L}_1-K_\Sigma,$ the norm changes by $e^{-\tilde\lambda}.$ 
Let's call the Chern connection on ${\mathcal{L}}_1$ with respect to the new metric ${\tilde{h}}_1$ by $B_0'.$ We write equations \eqref{curature 2 - 1sigma},\eqref{curature 2 - 2sigma} with respect to the Chern connection $B_0'$ and the metric $\tilde{h}_1$ on $\mathcal{L}_1:$
\begin{align}
    2\Delta_\Sigma(-\text{Im}f)+\frac{1}{4}e^{-4\text{Im}f}|\xi|^2_{\tilde{h}_1}=r_1\label{850}\\
    2F_{B_0'}-F_{K_\Sigma}=-\frac{i}{4}e^{-4\text{Im}f}|\xi|^2_{\tilde{h}_1}\omega_\Sigma
\end{align}
Since $F_{B'}=F_B-\partial\bar\partial\tilde{\lambda},$ the last equation reads
\begin{align}\label{curvature 2 - 4sigma}
    2F_{B_0}-F_{K_\Sigma}-2\partial\bar\partial\tilde{\lambda}=-\frac{i}{4}e^{-4\text{Im}f-\tilde{\lambda}}|\xi|^2\omega_\Sigma
\end{align}
Taking inner product with $\omega_{\Sigma}$ on both sides of the equation \eqref{curvature 2 - 4sigma} we get
\begin{align*}
    \langle 2F_{B_0}-F_{K_\Sigma},\omega_\Sigma\rangle-i\Delta_\Sigma\tilde\lambda=-\frac{i}{4}e^{-4\text{Im}f-\tilde{\lambda}}|\xi|^2
\end{align*}
which is same as
\begin{align}\label{curvature 2 - 5sigma}
    \Delta_\Sigma(-\tilde\lambda)+\frac{1}{4}e^{-4\text{Im}f-\tilde{\lambda}}|\xi|^2 =i\langle (2F_{B_0}-F_{K_\Sigma}),\omega_\Sigma\rangle
\end{align}
Multiplying both sides of the equation \eqref{850} by $2$ and adding it to the equation \eqref{curvature 2 - 5sigma} we get
\begin{align}\label{64}
    \Delta_\Sigma(-4\text{Im}f-\tilde\lambda)+\frac{3|\xi|^2}{4}e^{(-4\text{Im}f-\tilde\lambda)}=2r_1+i\langle (2F_{B_0}-F_{K_\Sigma}),\omega_\Sigma\rangle
\end{align}
$|\xi|^2$ has isolated zeroes and since $r_1=-\frac{2\pi\text{deg}(K_\Sigma-2\mathcal{L}_1)}{\text{vol}(\Sigma)}>0,$ we have $\int_\Sigma (2r_1+i\langle (2F_{B_0}-F_{K_\Sigma}),\omega_\Sigma\rangle)\omega_\Sigma>0.$ Hence we have a unique solution for $(-4\text{Im}f-\tilde\lambda)$ from equation \eqref{64} \cite{KW}.\par
On the other hand, multiplying both sides of the equation \eqref{850} by $-1$ and adding it to the equation \eqref{curvature 2 - 5sigma} we get 
\begin{align}
    \Delta_\Sigma(2\text{Im}f-\tilde{\lambda})=-r_1+i\langle (2F_{B_0}-F_{K_\Sigma}),\omega_\Sigma\rangle\label{65}
\end{align}
Since $r_1=-\frac{2\pi\text{deg}(K_\Sigma-2\mathcal{L}_1)}{\text{vol}(\Sigma)},\int_\Sigma (-r_1+i\langle (2F_{B_0}-F_{K_\Sigma}),\omega_\Sigma\rangle)\omega_\Sigma=0.$ Thereafter there is a unique solution for $(2\text{Im}f-\tilde{\lambda})$ from equation \eqref{65}. Since we we have solved both the functions $-(4\text{Im}f+\tilde\lambda),(2\text{Im}f-\tilde\lambda),$ we can determine Im$f$ and $\tilde\lambda$ individually. Finally we make a $\mathbb{C}^*$ gauge transformation to get a solution with respect to the initial metric.  
\end{proof}
\begin{remark}
    From the proof above, observe that for fixed holomorphic structures on $\mathcal{L}_0$ and $\mathcal{L}_1$, the construction relies solely on the non-trivial holomorphic sections of \(\mathcal{L}_0\) and \(K_\Sigma - \mathcal{L}_1\). Similar reasoning as in \S\ref{vortices} demonstrates that two pairs of holomorphic sections of \(\mathcal{L}_0\) and \(K_\Sigma - \mathcal{L}_1\) belonging to the same conformal class (i.e., differing by constant multiples) yield gauge-equivalent solutions to the Seiberg--Witten equations \eqref{Dirac 6d + sigma}, \eqref{Curvature 6d + sigma}, \eqref{Dirac 6d - sigma}, and \eqref{Curvature 6d - sigma}. Hence, the moduli space of the constructed solutions can be identified with $\mathcal{M}_{\text{vortex}}\big(c_1(\mathcal{L}_0)\big)\times \mathcal{M}_{\text{vortex}}\big(c_1(K_\Sigma\otimes\mathcal{L}_1^{-1})\big).$
\end{remark}
\section{Solutions to the $8$-dimensional Seiberg--Witten equations}\label{8d}
In this section we construct solutions to the $8$-dimensional Seiberg--Witten equations \eqref{Dirac 8d}, \eqref{Curvature 8d} on a closed K\"ahler $4$-fold $(X,\omega).$ 
\subsection{Spin$^\mathbb{C}$-bundle, Dirac operator and Clifford multiplication}
For a holomorphic Hermitian line bundle \(\mathcal{L}\), we define the positive spin bundle as
\begin{align*}
    S_+(X)=\Lambda^0(X,\mathcal{L})\oplus \Lambda^{0,2}(X,\mathcal{L})\oplus \Lambda^{0,4}(X,\mathcal{L}).
\end{align*}
The associated line bundle to the spin$^\mathbb{C}$-structure is
\begin{align*}
    L(S)=K_X^{-1}\otimes\mathcal{L}^2.
\end{align*}
A unitary connection $A$ on $L(S)$ is equivalent to a unitary connection $A_0$ on $\mathcal{L}_0,$ the equivalence being $A=-A_{K_X}+2A_0$ (with abuse of notation) where $A_{K_X}$ is the Chern connection on $K_X$ with respect to the canonical hermitian metric on $K_X$ induced from the K\"ahler metric. \par 
The Dirac operator associated to the connection $A$ on $L(S)$ is\cite{Morgan}
\begin{align*}
    D_A=\sqrt{2}(\bar\partial_{A_0}+{\bar\partial}_{A_0}^*):\Omega^0(X,\mathcal{L})\oplus \Omega^{0,2}(X,\mathcal{L})\oplus \Omega^{0,4}(X,\mathcal{L})\rightarrow \Omega^{0,1}(X,\mathcal{L})\oplus \Omega^{0,3}(X,\mathcal{L}).
\end{align*}
The operator is obtained by coupling $\sqrt{2}(\bar\partial+{\bar\partial}^*)$ with the covariant derivative $\nabla_{A_0}.$\par
We prove some formulae for Clifford action of two, three and four forms on the positive spinors. These formulae will be explicitly used in the construction of a solution of the $8$-dimensional Seiberg--Witten equations \eqref{Dirac 8d}, \eqref{Curvature 8d}.
\begin{lemma}\label{Clifford action lemma 1 8d} For any $\eta\in\Omega^{0,1},\varphi\in\Omega^0(X,\mathcal{L})$ and $\xi\in\Omega^{0,4}(X,\mathcal{L}),$ we have
\begin{align*}
    &1.\hspace{1 ex}c(\eta\wedge\omega)\varphi=-3\sqrt{2}i\eta\wedge\varphi\\
    &2.\hspace{1 ex} c(\bar\eta\wedge\omega)\xi=3\sqrt{2}*(\bar\eta\wedge\xi)
\end{align*}
\end{lemma}
\begin{proof}
    We take local holomorphic coordinates $\{z_k=x_k+iy_k\}_{k=1,2,3,4}$ centered at a point $x\in X$ so that the Kähler metric is standard to second order at the point. All the calculations done in this proof are at this point $x$ with respect to the chosen local coordinates. Let's take $j,k\in\{1,2,3,4\}.$ Using linearity of the Clifford action enough to prove the formulae when $\eta=d\bar z_j$ at $x$. Using formula \ref{Clifford one form} we compute:
    \begin{align*}
        c(d\bar z_j\wedge\omega)\varphi&=c\big(d\bar z_j\wedge(\sum_{k\neq j}dx_k\wedge dy_k)\big)\varphi\\
        &=c(d\bar z_j)\circ c\big(\sum_{k\neq j}dx_k\wedge dy_k\big)\varphi\\
        &=c(d\bar z_j)(-3i\varphi)\\
        &=-3\sqrt{2}id\bar z_j\wedge\varphi
    \end{align*}
    and 
    \begin{align*}
        c(dz_j\wedge\omega)\xi&=c\big(dz_j\wedge(\sum_{k\neq j}dx_k\wedge dy_k)\big)xi\\
        &=c(dz_j)\circ c\big(\sum_{k\neq j}dx_k\wedge dy_k\big)\xi\\
        &=c(dz_j)(3i\xi)\\
        &=-3\sqrt{2}i(d\bar z_j\angle\xi)\\
        &=-3\sqrt{2}i*(dz_j\wedge\xi)
    \end{align*}
Hence, the lemma follows.    
\end{proof}
We also describe the Clifford action of the forms $i\omega$ and $\omega^2$ on the positive spinors. The proof of the following lemma involves very similar computations as in the lemma \ref{Clifford omega} and is left to the reader to fill in the details.
\begin{lemma}\label{Clifford action lemma 2 8d}We have the positive spin bundle as the summand $S_+=\mathcal{L}\oplus(\Lambda^{0,2}\otimes\mathcal{L})\oplus(\Lambda^{0,4}\otimes\mathcal{L}).$ $ c(i\omega)$ and $c(\omega^2)$ preserve the splitting of $S_+$ and act on the summand as follows:
    \begin{align*}
        c(i\omega)=\begin{bmatrix}
            4\text{Id}&0&0\\
            0&0&0\\
            0&0&-4\text{Id}
        \end{bmatrix} \text{ and }
        c(\omega^2)=\begin{bmatrix}
            -12\text{Id}&0&0\\
            0&4\text{Id}&0\\
            0&0&12\text{Id}
        \end{bmatrix}  
    \end{align*}
\end{lemma}
\subsection{Solution on closed K\"ahler $4$-folds}\label{Solution 8d}
Before we explain the construction of a solution of the Seiberg--Witten equations \eqref{Dirac 8d},\eqref{Curvature 8d}, we will need two lemmas from K\"ahler geometry.
\begin{lemma}\label{lemma Kahler 8d 1st}For a smooth complex valued function $h$ on a Kähler $4$-fold $(X,\omega),$ we have 
\begin{equation*}
    *(\partial\bar\partial h\wedge\omega)=-\partial\bar\partial h\wedge\omega+\frac{i}{4}\Delta(h)\omega^2
\end{equation*}
\begin{proof}
    We take local holomorphic coordinates $\{z_k=x_k +i y_k\}_{k=1,2,3,4}$ centered at a point $x\in X$ so that the Kähler metric is standard to second order at the point. All calculations in the following are performed at the point $x$. We compute
    \begin{align*}
    &(\partial\bar\partial h)\wedge\omega\\
    &=\sum_{j,k=1}^4 \frac{\partial^2h}{\partial z_j\partial \bar{z}_k}d z_j\wedge d\bar{z}_k\wedge\omega\\
    &=\sum_{j\neq k}\frac{\partial^2h}{\partial z_j\partial \bar{z}_k}d z_j\wedge d\bar{z}_k\wedge\omega+\sum_{j=1}^4 \frac{\partial^2h}{\partial z_j\partial \bar{z}_j}d z_j\wedge d\bar{z}_j\wedge\omega
\end{align*}
A small calculation shows that for $j\neq k,$
\begin{align*}
    *(d z_j\wedge d\bar{z}_k\wedge\omega)=-(d z_j\wedge d\bar{z}_k\wedge\omega)
\end{align*}
and \begin{align*}
    *(dz_k\wedge dz_k\wedge\omega)=-(dz_k\wedge dz_k\wedge\omega)-i\omega^2
\end{align*}
Thereafter we get
\begin{align*}
    &*\big((\partial\bar\partial h)\wedge\omega\big)\\
    &=*\bigg(\sum_{j\neq k}\frac{\partial^2h}{\partial z_j\partial \bar{z}_k}d z_j\wedge d\bar{z}_k\wedge\omega+\sum_{j=1}^4 \frac{\partial^2h}{\partial z_j\partial \bar{z}_j}d z_j\wedge d\bar{z}_j\wedge\omega\bigg)\\
    &=-\sum_{j\neq k}\frac{\partial^2h}{\partial z_j\partial \bar{z}_k}d z_j\wedge d\bar{z}_k\wedge\omega-\sum_{j=1}^4 \frac{\partial^2h}{\partial z_j\partial \bar{z}_j}d z_j\wedge d\bar{z}_j\wedge\omega-i\big(\sum_{j=1}^4 \frac{\partial^2h}{\partial z_j\partial \bar{z}_j}\big)\omega^2\\
    &=-(\partial\bar\partial h)\wedge\omega-\frac{i}{4}\Big(\sum\limits_{k=1}^{4}\Big(\frac{\partial^2 h}{\partial x_k^2}+\frac{\partial^2 h}{\partial y_k^2}\Big)\Big)\omega^2\\
    &=-(\partial\bar\partial h)\wedge\omega+\frac{i}{4}\Delta(h)\omega^2\tag*{\qedhere} 
\end{align*}
\end{proof}
    
\end{lemma}
\begin{lemma}\label{lemma Kahler 8d 2nd}
     For a smooth real-valued function $h$ on a Kähler $4$-fold $(X,\omega),$ we have
    \begin{align*}
        *\big(\partial\bar\partial h\wedge\omega^2\big)=-2\partial\bar\partial h+i\Delta (h)\omega
    \end{align*}
\end{lemma}
\begin{proof}
    We take local holomorphic coordinates $\{z_k=x_k +i y_k\}_{k=1,2,3,4}$ centered at a point $x\in X$ so that the Kähler metric is standard to second order at the point. All calculations in the following are performed at the point $x$. We compute
\begin{align*}
    *\big(\partial\bar\partial h\wedge\omega^2\big)&=*\Big(\sum\limits_{j,k=1}^{4} \frac{\partial^2h}{\partial z_j\partial\bar{z}_k}d z_j\wedge d\bar{z}_k\wedge\omega^2\Big)\\
    &=*\Big(\sum\limits_{j\neq k}^{} \frac{\partial^2h}{\partial z_j\partial\bar{z}_k}d z_j\wedge d\bar{z}_k\wedge\omega^2+\sum_{k=1}^4 \frac{\partial^2h}{\partial z_k\partial\bar{z}_k}d z_k\wedge d\bar{z}_k\wedge\omega^2\Big)
\end{align*}
A small calculation shows that for $j\neq k,$
\begin{align*}
    *(d z_j\wedge d\bar{z}_k\wedge\omega^2)=-2d z_j\wedge d\bar{z}_k
\end{align*}
and 
\begin{align*}
    *\big(d z_k\wedge d\bar{z}_k\wedge\omega^2\big)&=-2d z_k\wedge d\bar{z}_k-4i\omega
\end{align*}
Hence we get
\begin{align*}
    &*\Big(\sum\limits_{j\neq k}^{} \frac{\partial^2h}{\partial z_j\partial\bar{z}_k}d z_j\wedge d\bar{z}_k\wedge\omega^2+\sum_{k=1}^3 \frac{\partial^2h}{\partial z_k\partial\bar{z}_k}d z_k\wedge d\bar{z}_k\wedge\omega^2\Big)\\
    &=-2\sum\limits_{j\neq k} \frac{\partial^2h}{\partial z_j\partial\bar{z}_k}d z_j\wedge d\bar{z}_k-2\sum_{k=1}^4 \frac{\partial^2h}{\partial z_k\partial\bar{z}_k}d z_k\wedge d\bar{z}_k-4i\Big(\sum_{k=1}^4 \frac{\partial^2h}{\partial z_k\partial\bar{z}_k}\Big)\omega\\
    &=-2\Big(\sum\limits_{j,k=1}^{4} \frac{\partial^2h}{\partial z_j\partial\bar{z}_k}d z_j\wedge d\bar{z}_k\Big)-i\Big(\sum\limits_{k=1}^{4}\Big(\frac{\partial^2 h}{\partial x_k^2}+\frac{\partial^2 h}{\partial y_k^2}\Big)\Big)\omega \\
    &=-2\partial\bar\partial h+{i}\Delta (h)\omega\tag*{\qedhere} 
\end{align*}
\end{proof}
\begin{proof}[Proof of theorem \ref{theorem 8d}] Say $h$ be the hermitian metric on $\mathcal{L}.$ We call the Chern connection on $\mathcal{L}$ with respect to $h$ by $A_0.$ In the first part of the proof we assume that the first two conditions of theorem \ref{theorem 8d} are met, i.e., dim $H^0(X,\mathcal{L})>0$ and $c_1(K_X^{-1}\otimes\mathcal{L}^2)=a[\omega]$ for $a<0.$\par
We take a nontrivial holomorphic section $\varphi$ of $\mathcal{L}.$ Hence, we have $\bar\partial_{A_0}\varphi=0.$ It follows from the lemma \ref{Clifford action lemma 1 8d} that for any $f\in C^\infty(X,\mathbb{C}),\phi=e^{3i(1+i)f}\varphi,A=-A_{K_X}+2A_0$ and $\beta=(\bar\partial f+\partial \bar f)\wedge\omega$ solves the Dirac equation \eqref{Dirac 8d}:
\begin{align*}
    \big(D_A+(1+i)c(\beta)\big)\phi=\sqrt{2}\bar\partial_{A_0}\phi+c(\bar\partial f\wedge\omega)\phi=0
\end{align*}
Next we compute the terms $d\beta^+$ and $d^*\beta$ for our choice of $\beta=(\bar\partial f+\partial \bar f)\wedge\omega$.
\begin{align*}
    d\beta=(\partial\bar\partial f+\bar\partial\partial\bar f)\wedge\omega=2i\partial\bar\partial(\text{Im}f)\wedge\omega
\end{align*}
Hence using lemma \ref{lemma Kahler 8d 1st} we get
\begin{align*}
    d\beta^+&=\frac{1}{2}\big(d\beta+*(d\beta)\big)\\
    &=i\big(\partial\bar\partial(\text{Im}f)\wedge\omega-(\partial\bar\partial(\text{Im}f)\wedge\omega)+\frac{i}{4}\Delta(h)\omega^2\big)\\
    &=-\frac{1}{4}\Delta(h)\omega^2
\end{align*}
To compute $d^*\beta,$ we take local holomorphic coordinates $\{z_k=x_k+iy_k\}_{k=1,2,3,4}$ centered at a point $x\in X$ so that the Kähler metric is standard to second order at the point. In these coordinates notice that for any $k\in\{1,2,3,4\},$
\begin{align*}
    *(d\bar z_k\wedge\omega)=\frac{i}{2}(d\bar z_k\wedge\omega^2)\text{ and } *(dz_k\wedge\omega)=-\frac{i}{2}(dz_k\wedge\omega^2)
\end{align*}
Therefore,
\begin{align*}
    *\beta&=*\Big(\sum_{j=1}^4\big(\frac{\partial f}{\partial\bar z_j}d\bar z_j+\frac{\partial \bar f}{\partial z_j}dz_j\big)\wedge\omega\Big)\\
    &=\frac{i}{2}\sum_{j=1}^4\big(\frac{\partial f}{\partial\bar z_j}d\bar z_j-\frac{\partial \bar f}{\partial z_j}dz_j\big)\wedge\omega^2\\
    &=\frac{i}{2}(\bar\partial f-\partial\bar f)\wedge\omega^2
\end{align*}
Finally use lemma \ref{lemma Kahler 8d 2nd} to calculate
\begin{align*}
    d^*\beta&=-*d\Big(\frac{i}{2}\big(\bar\partial f-\partial\bar f\big)\wedge\omega^2\Big)\\
    &=-i*\big(\partial\bar\partial(\text{Re}f)\wedge\omega^2\big)\\
    &=2i\partial\bar\partial(\text{Re}f)+\Delta(\text{Re}f)\omega
\end{align*}
Next we shall identify the quadratic term $q(\phi)$ in terms of $i\omega$ and $\omega^2$. Since, $\phi\in\Omega^0(X,\mathcal{L}), c\big(q(\phi)\big)$ preserves the splitting of $S_+=\mathcal{L}\oplus(\Lambda^{0,2}\otimes\mathcal{L})\oplus (\Lambda^{0,4}\otimes\mathcal{L})$ and acts on the summand as 
\begin{align*}
    c\big(q(\phi)\big)=\begin{bmatrix}
        \frac{7}{8}|\phi|^2\text{Id}& 0&0\\
        0&-\frac{1}{8}|\phi|^2\text{Id}&0\\
        0&0&-\frac{1}{8}|\phi|^2\text{Id}
    \end{bmatrix}
\end{align*}
Using lemma \ref{Clifford action lemma 2 8d}, we identify the quadratic term to be:
\begin{align}
    q(\phi)&=\frac{|\phi|^2}{32}(4i\omega-\omega^2)\\
    &=\frac{|\varphi|^2}{32}e^{-6(\text{Re}f+\text{Im}f)}(4i\omega-\omega^2)
\end{align}
Using all these formulae, we write down the simplified version of the curvature equation \eqref{Curvature 8d}. We write down the $i\Omega^2$ part and the $\Omega^4_+$ part of the equation separately.
\begin{align}
    F_A-4\partial\bar\partial(\text{Re}f)+2i\Delta(\text{Re}f)\omega=\frac{i}{8}e^{-6(\text{Re}f+\text{Im}f)}|\varphi|^2\omega \label{Curvature 8d 1}\\
    -\frac{1}{2}\Delta(\text{Im}f)+a=-\frac{1}{32}e^{-6(\text{Re}f+\text{Im}f)}|\varphi|^2\label{Curvature 8d 2}
\end{align}
To solve these two equations, we make a conformal change to the metric $h$ on $\mathcal{L}$ by $e^{2\lambda}$ for a smooth function $\lambda:X\rightarrow\mathbb{R}.$ Call this new metric $h'=e^{2\lambda} h.$ Notice since the change in metric is conformal $\varphi$ is still a holomorphic section of $(\mathcal{L},h'),$ i.e., if $A_0'$ is the Chern connection on $\mathcal{L}$ with respect to $h',\bar\partial_{A_0'}\varphi=0.$ We call the new induced metric on $L(S)=K_X^{-1}\otimes\mathcal{L}^2$ by $A'$. With respect to $A'$ and $h'$ equations \eqref{Curvature 8d 1} and \eqref{Curvature 8d 2} read
\begin{align*}
    F_{A'}-4\partial\bar\partial(\text{Re}f)+2i\Delta(\text{Re}f)\omega=\frac{i}{8}e^{-6(\text{Re}f+\text{Im}f)}|\varphi|^2_{h'}\omega\\
    -\frac{1}{2}\Delta(\text{Im}f)+a=-\frac{1}{32}e^{-6(\text{Re}f+\text{Im}f)}|\varphi|_{h'}^2
\end{align*}
Since $F_{A'}=F_A-4\partial\bar\partial\lambda,$ these are same as
\begin{align}
    F_A-4\partial\bar\partial(\lambda+\text{Re}f)+2i\Delta(\text{Re}f)\omega=\frac{i}{8}e^{-6(\text{Re}f+\text{Im}f)+2\lambda}|\varphi|^2\omega\label{Curvature 8d 3}\\
    -\frac{1}{2}\Delta(\text{Im}f)+a=-\frac{1}{32}e^{-6(\text{Re}f+\text{Im}f)+2\lambda}|\varphi|^2
\end{align}
The condition $c_1(K_X^{-1}\otimes\mathcal{L}^2)=\frac{2a}{\pi}[\omega]$ gives us $F_A=4\partial\bar\partial f_0-4ai\omega$ for an $f_0\in C^\infty(X,\mathbb{R}).$
Hence equation \eqref{Curvature 8d 3} reads
\begin{align*}
    4\partial\bar\partial(f_0-\lambda-\text{Re}f)-4ai\omega+2i\Delta(\text{Re}f)\omega=\frac{i}{8}e^{-6(\text{Re}f+\text{Im}f)+2\lambda}|\varphi|^2\omega,
\end{align*}
which breaks down into the following two equations.
\begin{align}
    f_0=\lambda+\text{Re}f\label{Curvature 8d 4}\\
    \Delta(-\text{Re}f)+\frac{1}{16}e^{-6(\text{Re}f+\text{Im}f)+2\lambda}|\varphi|^2=-2a\label{Curvature 8d 5}
\end{align}
Replacing $\lambda$ from the equation \eqref{Curvature 8d 4} into the equations \eqref{Curvature 8d 3}, \eqref{Curvature 8d 5} we get
\begin{align}
    \Delta(-\text{Im}f)+\frac{1}{16}e^{(-8\text{Re}f-6\text{Im}f+2f_0)}|\varphi|^2=-2a\label{Curvature 8d similar 1}\\
    \Delta(-\text{Re}f)+\frac{1}{16}e^{(-8\text{Re}f-6\text{Im}f+2f_0)}|\varphi|^2=-2a\label{Curvature 8d similar 2}
\end{align}
These two equations lead us to 
\begin{align*}
    \Delta(\text{Re}f-\text{Im}f)=0.
\end{align*}
Which implies
\begin{align*}
    \text{Re}f=\text{Im}f+c_0
\end{align*}
for a real-valued constant $c_0.$ Moreover, the two equations \eqref{Curvature 8d similar 1}, \eqref{Curvature 8d similar 2} merge into a single equation:
\begin{align}
    \Delta(-\text{Re}f)+\frac{e^{(2f_0+6c_0)}|\varphi|^2}{16}e^{-14\text{Re}f}=-2a\label{Curvature 8d similar 3}
\end{align}
Since $|\varphi|^2$ has isolated zeroes and $-2a>0,$ equation \eqref{Curvature 8d similar 3} has a unique solution for Re$f$ \cite{KW}. Once we know Re$f,$ we solve $\lambda$ from equation \eqref{Curvature 8d 4}.\par
In the second part of the proof we show existence of a non-trivial solution of the Seiberg--Witten equations \eqref{Dirac 8d},\eqref{Curvature 8d} when the second pair of conditions in theorem \ref{theorem 8d} are satisfied, i.e., $\text{dim }H^0(X,K_X\otimes\mathcal{L}^{-1})>0 \text{ and }c_1(K_X^{-1}\otimes \mathcal{L}^2)=-\frac{2a}{\pi}[\omega]$ with $a<0.$\par
We take a non-trivial section $\xi\in\Omega^{0,4}(X,\mathcal{L})\cong \Omega^0(X,K_X^{-1}\otimes \mathcal{L})$ such that $\bar\partial_{A_0}^*\xi=0.$ This is equivalent to choosing a non-trivial holomorphic section $\bar\xi$ of $\overline{K_X^{-1}\otimes\mathcal{L}}\cong (K_X^{-1}\otimes\mathcal{L})^*,$ so we are indeed using the assumption that dim $H^0(X,K_X\otimes\mathcal{L}^{-1})> 0.$ For $f_1\in C^\infty(X,\mathbb{C}),$ define $\phi:=e^{-3i(1+i)\bar f_1}\xi$. A small calculation gives us 
\begin{align}
    \bar\partial_{A_0}^*\phi=\bar\partial_{A_0}^*(e^{-3i(1+i)\bar f_1}\xi)=3i(1+i)*(\partial\bar f_1\wedge\phi)\label{Calculation 1}
\end{align}
Thereafter, using lemma \ref{Clifford action lemma 1 8d} notice that $\phi=e^{-3i(1+i)\bar f_1}\xi, A=-A_{K_X}+2A_0,\beta=(\bar\partial f_1+\partial\bar f_1)\wedge\omega$ solves the Dirac equation \eqref{Dirac 8d}:
\begin{align*}
    \big(D_A+(1+i)c(\beta)\big)(\phi)=\sqrt{2}\bar\partial_{A_0}^*\phi+(1+i)c(\partial\bar f_1\wedge\omega)\phi=0
\end{align*}
Using lemma \ref{Clifford action lemma 2 8d} we identify the quadratic term $q(\phi)$ in the similar fashion as in the first part of the proof. We deduce that for $\phi\in\Omega^{0,4}(X,\mathcal{L}),$
\begin{align}
    q(\phi)&=-\frac{|\phi|^2}{32}(4i\omega+\omega^2)\\
    &=-\frac{e^{6(\text{Re}f_1-\text{Im}f_1)}|\xi|^2}{32}(4i\omega+\omega^2)
\end{align}
We write down the $i\Omega^2$ and the $\Omega^4_+$ parts of the curvature equation \eqref{Curvature 8d} separately.
\begin{align}
    F_A-4\partial\bar\partial(\text{Re}f_1)+2i\Delta(\text{Re}f_1)\omega=-\frac{i}{8}e^{6(\text{Re}f_1-\text{Im}f_1)}|\xi|^2\omega\label{Curvature 8d 2 1}\\
    -\frac{1}{2}\Delta(\text{Im}f_1)+a=-\frac{1}{32}e^{6(\text{Re}f_1-\text{Im}f_1)}|\xi|^2\label{Curvature 8d 2 2}
\end{align}
To solve these two equations, we perturb the initial metric $h$ on ${\mathcal{L}}$ by $e^{2\tilde\lambda},\tilde\lambda\in C^\infty(X,\mathbb{R}).$ The new metric being $\tilde{h}:=e^{2\tilde{\lambda}} h.$ Notice this conformal change in the metric doesn't change the holomorphic structure of $\mathcal{L}$ and induces a conformal change by $e^{-2\tilde\lambda}$ on the corresponding hermitian metric $\mathcal{L}^*\cong \mathcal{L}^{-1}.$ So, $\bar\xi$ remains a holomorphic section of $\overline{K_X^{-1}\otimes\mathcal{L}}\cong K_X\otimes\mathcal{L}^{-1},$ and now when we go back to the corresponding anti-holomorphic section on $K_X^{-1}\otimes\mathcal{L},$ the norm changes by $e^{-2\tilde\lambda}.$ 
Let's call the Chern connection on ${\mathcal{L}}$ with respect to the new metric ${\tilde{h}}$ by $\tilde{A_0}$ and the new induced metric on $L(S)=K_X^{-1}\otimes\mathcal{L}^2$ by $\tilde{A}=-A_{K_X}+2\tilde{A_0}.$ In this new set up the curvature equations \eqref{Curvature 8d 2 1} and \eqref{Curvature 8d 2 2} read
\begin{align*}
    F_{\tilde{A}}-4\partial\bar\partial(\text{Re}f_1)+2i\Delta(\text{Re}f_1)\omega=-\frac{i}{8}e^{6(\text{Re}f_1-\text{Im}f_1)}|\xi|_{\tilde{h}}^2\omega\\
    -\frac{1}{2}\Delta(\text{Im}f_1)+a=-\frac{1}{32}e^{6(\text{Re}f_1-\text{Im}f_1)}|\xi|_{\tilde{h}}^2
\end{align*}
Since $F_{\tilde{A}}=F_A-4\partial\bar\partial\tilde\lambda,$ these equations are same as
\begin{align}
    F_A-4\partial\bar\partial(\tilde\lambda+\text{Re}f_1)+2i\Delta(\text{Re}f_1)\omega=-\frac{i}{8}e^{6(\text{Re}f_1-\text{Im}f_1)-2\tilde\lambda}|\xi|^2\omega\label{Random 1}\\
    -\frac{1}{2}\Delta(\text{Im}f_1)+a=-\frac{1}{32}e^{6(\text{Re}f_1-\text{Im}f_1)-2\tilde\lambda}|\xi|^2\label{Random 2}
\end{align}
The condition $c_1(K_X^{-1}\otimes\mathcal{L}^2)=-\frac{2a}{\pi}[\omega]$ gives us $F_A=4\partial\bar\partial g_1+4ai\omega$ for an $g_1\in C^\infty(X,\mathbb{R}).$ Hence equation \eqref{Random 1} reads
\begin{align*}
    4\partial\bar\partial(g_1-\tilde\lambda-\text{Re}f_1)+4ai\omega+2i\Delta(\text{Re}f_1)\omega=-\frac{i}{8}e^{6(\text{Re}f_1-\text{Im}f_1)-2\tilde\lambda}|\xi|^2\omega
\end{align*}
which breaks down into the following two equations
\begin{align}
    g_1=\tilde\lambda+\text{Re}f_1\label{random 1}\\
    \Delta(\text{Re}f_1)+\frac{1}{16}e^{6(\text{Re}f_1-\text{Im}f_1)-2\tilde\lambda}|\xi|^2=-2a\label{random 2}
\end{align}
Replacing $\tilde\lambda$ from equation \eqref{random 1} into equations \eqref{random 2} and \eqref{Random 2} we get
\begin{align}
    \Delta(-\text{Im}f_1)+\frac{1}{16}e^{(8\text{Re}f_1-6\text{Im}f_1-2g_1)}|\xi|^2=-2a\label{eq 1}\\
    \Delta(\text{Re}f_1)+\frac{1}{16}e^{(8\text{Re}f_1-6\text{Im}f_1-2g_1)}|\xi|^2=-2a\label{eq 2}
\end{align}
These two equations lead us to
\begin{align}
    \Delta(\text{Re}f_1+\text{Im}f_1)=0
\end{align}
which implies
\begin{align*}
    \text{Re}f_1=-\text{Im}f_1-d_0
\end{align*}
for a real-valued constant $d_0.$ Moreover, the two equations \eqref{eq 1}, \eqref{eq 2} merge into a single equation:
\begin{align}
    \Delta(\text{Re}f_1)+\frac{e^{(6d_0-2g_1)}|\xi|^2}{16}e^{14\text{Re}f_1}=-2a\label{eq 3}
\end{align}
Since $|\xi|^2$ has isolated zeroes and $-2a>0,$ there exists of unique solution of Re$f_1$ from \eqref{eq 3} \cite{KW}. Once we know Re$f_1$, we get $\tilde\lambda$ from equation \eqref{random 1}.
\end{proof}
\subsection{Relationship with vortices}
Let's say for a holomorphic hermitian line bundle $\mathcal{L}$, the first two conditions in theorem \ref{theorem 8d} are satisfied. From the above proof, it follows that there exists a hermitian metric on $\mathcal{L}$ such that for any non-trivial holomorphic section $\varphi\in H^0(X,\mathcal{L}),$ there exists a unique $f\in C^\infty(X,\mathbb{R})$ and a unitary connection $A_0$ on $\mathcal{L}$ such that
    \begin{align}
        \big(\phi=e^{3i(1+i)f}\varphi,A=-A_{K_X}+2A_0,\beta=(\bar\partial f+\partial\bar f)\wedge\omega\big)\label{soln 1}
    \end{align}
    solves the Seiberg--Witten equations \eqref{Dirac 8d}, \eqref{Curvature 8d}.\par 
    For a non-zero complex-valued constant $c,$ define
    \begin{align*}
        f_c:=f+\frac{(1+i)}{6}\text{ln}|c|
    \end{align*}
    Instead of $\varphi,$ if we would have started with $c\varphi,$ the construction explained in the proof of theorem \ref{theorem 8d}, would give us the solution 
    \begin{align}
        \big(\phi=e^{3i(1+i)f_c}c\varphi,A=-A_{K_X}+2A_0,\beta=(\bar\partial f_c+\partial\bar f_c)\wedge\omega\big)\label{soln 2}
    \end{align}
    Notice 
    \begin{align*}
        e^{3i(1+i)f_c}c\varphi=e^{3i(1+i)\times \frac{(1+i)}{6}\text{ln}|c|}c \times(e^{3i(1+i)f}\varphi)=\frac{c}{|c|}e^{3i(1+i)f}\varphi\\
        \text{and } (\bar\partial f_c+\partial\bar f_c)\wedge\omega=(\bar\partial f+\partial\bar f)\wedge\omega
    \end{align*}
    Hence the solution \ref{soln 2} coming from $c\varphi$ is gauge equivalent to the solution \ref{soln 1} coming from $\varphi.$ Therefore, we can identify the moduli space of constructed solutions in theorem \ref{theorem 8d} with $\mathcal{M}_{\text{vortex}}(\mathcal{L})$. If $\mathcal{L}$ satisfies the second pair of conditions from theorem \ref{theorem 8d}, similar reasoning as above would imply that the vortices with respect to the line bundle $K_X\otimes\mathcal{L}^{-1}$ also lead to solutions of the Seiberg--Witten equations and the space of constructed solutions can be identified with $ \mathcal{M}_{\text{vortex}}(K_X\otimes\mathcal{L}^{-1}).$ If $b_1=0,$ notice $\mathcal{M}_{\text{vortex}}(\mathcal{L})\cong\mathbb{CP}\big(H^0(X,\mathcal{L})\big).$ Thereafter for $b_1=0$, if $\mathcal{L}$ is the trivial bundle satisfying the first two conditions or the canonical bundle satisfying the alternative two conditions, the moduli space of constructed solutions is a singleton set. 

\section{Solutions to the $5$-dimensional Seiberg--Witten equations}\label{5d}
This section explains the construction of a non-trivial solution to the  Seiberg--Witten equations \eqref{Dirac 5D perturbed},\eqref{Curvature 5D perturbed}. Throughout this section our base manifold is $\Sigma\times\mathbb{R}^3,$ with the product metric on it. 
\subsection{Spin$^\mathbb{C}$-bundles, Dirac operator and Clifford multiplication } 
Following \ref{Spinor in 5d}, we take the spinor bundle on $\Sigma\times\mathbb{R}^3$ as:
\begin{align}
    S(\Sigma\times\mathbb{R}^3)=\big(\pi_1^*(\Lambda^0(\Sigma,\mathcal{L}))\otimes\pi_2^*(\underline{\mathbb{C}}^2)\big)\oplus \big(\pi_1^*(\Lambda^{0,1}(\Sigma,\mathcal{L}))\otimes\pi_2^*(\underline{\mathbb{C}}^2)\big)
\end{align}
The associated line bundle is $L(S)=\pi_1^*(K_\Sigma^{-1}\otimes\mathcal{L}^2).$
$K_\Sigma$ being the canonical bundle of $\Sigma.$ The K\"ahler metric induces a metric on $K_\Sigma,$ let's call the Chern connection on $K_\Sigma$ with respect to this metric by $A_\Sigma.$ Now given a unitary connection on $A_{\mathcal{L}}$ on $\mathcal{L},$ we get a unitary connection on $K_\Sigma^{-1}\otimes\mathcal{L}^2$ and hence on $L(S)=\pi_1^*(K_\Sigma^{-1}\otimes\mathcal{L}^2)$ by pull back. With abuse of notation we call it $A=\pi_1^*(-A_\Sigma+2A_{\mathcal{L}}).$\par
Let's take $e_1,e_2$ to be the standard basis of $\mathbb{C}^2.$ If we take the standard flat metric and the flat connection on $\underline{\mathbb{C}}^2\rightarrow\mathbb{R}^3,$ then $e_1,e_2$ induce two unit-length nowhere vanishing parallel spinors on $\mathbb{R}^3.$ With abuse of notation we call them again by $e_1$ and $e_2.$ We follow the following convention for Clifford actions of the one-forms $dx_1,dx_2,dx_3$ on $e_1,e_2.$
\begin{align*}
    c(d x_1)=\begin{bmatrix}
        0&-i\\
        -i&0
    \end{bmatrix},
    c(d x_2)=\begin{bmatrix}
        0&-1\\
        -1&0
    \end{bmatrix},
    c(d x_3)=\begin{bmatrix}
        -i&0\\
        0&i
    \end{bmatrix}
\end{align*}
Thereafter the spinor bundle splits as the direct sum of four line bundles:
\begin{align}
    S(\Sigma\times\mathbb{R}^3)=&\big(\pi_1^*\Lambda^0(\Sigma,\mathcal{L})\otimes\pi_2^*\langle e_1\rangle\big)\oplus \big(\pi_1^*(\Lambda^0(\Sigma,\mathcal{L})\otimes\pi_2^*\langle e_2\rangle\big)\nonumber\\
    &\oplus\big(\pi_1^*\Lambda^{0,1}(\Sigma,\mathcal{L})\otimes\pi_2^*\langle e_1\rangle\big)\oplus \big(\pi_1^*\Lambda^{0,1}(\Sigma,\mathcal{L})\otimes\pi_2^*\langle e_2\rangle\big)\label{Splitting 5d}
\end{align}
We compute the Clifford action of the forms $i\pi_1^*\omega,i\pi_2^*(dx_1\wedge dx_2)$ and $\pi_1^*\omega\wedge \pi_2^*(dx_1\wedge dx_2).$ This will be used in the construction of the solution to the equations \eqref{Dirac 5D perturbed}, \eqref{Curvature 5D perturbed} to identify the quadratic term $q(\phi)$ in terms of the these forms. The Clifford action of the forms preserve the splitting \ref{Splitting 5d} of $S(\Sigma\times\mathbb{R}^3)$ and they act on the summands as follows:
\begin{align}
c(i\pi_1^*\omega)=\begin{bmatrix}
        \text{Id}&0&0&0\\
        0&\text{Id}&0&0\\
        0&0&-\text{Id}&0\\
        0&0&0&-\text{Id}
    \end{bmatrix},
    c(i\pi_2^*(dx_1\wedge dx_2))=\begin{bmatrix}
        \text{Id}&0&0&0\\
        0&-\text{Id}&0&0\\
        0&0&\text{Id}&0\\
        0&0&0&-\text{Id}
    \end{bmatrix}\label{Clifford formula 5d 1}\\
    c(\pi_1^*\omega\wedge \pi_2^*(dx_1\wedge dx_2))=\begin{bmatrix}
        -\text{Id}&0&0&0\\
        0&\text{Id}&0&0\\
        0&0&\text{Id}&0\\
        0&0&0&-\text{Id}\label{Clifford formula 5d 2}
    \end{bmatrix}
\end{align}

\subsection{Solution on $\Sigma\times\mathbb{R}^3$}\label{Soln 5d}
\begin{proof}[Proof of theorem \ref{Theorem 5d}] We split the proof into $4$ cases depending on the \textit{``perturbation"} term $\eta.$ We will see that for these $4$ different cases, the spinor $\phi$ will lie in $4$ different summands of $S(\Sigma\times\mathbb{R}^3).$ We call the Chern connection on $\mathcal{L}$ by $A_h$ for the prescribed hermitian metric $h$ on $\mathcal{L}.$\par 
\setul{1.2ex}{}
\ul{$\eta= a(id x_1\wedge d x_2-\omega\wedge d x_1\wedge d x_2),a=\frac{2\pi}{\text{vol}(\Sigma)}\text{deg}(K_\Sigma-2\mathcal{L})>0:$} 
Take $f_1\in C^\infty(\Sigma,\mathbb{C})$ and choose $\beta_3=\pi_1^*(\bar\partial f_1+\partial\bar f_1)\wedge \pi_2^*(dx_1\wedge dx_2)$ and $\beta_5=0.$\par
The assumption (from condition \ref{condition 1 5d}): dim  $H^0(\Sigma,\mathcal{L})>0$ lets us choose a non-trivial holomorphic section $\varphi$ of $\mathcal{L},$ i.e., $\bar\partial_{A_h}\varphi=0.$ Define a spinor $\phi:=\pi_1^*(e^{i(1-i)f_1}\varphi)\otimes\pi_2^*e_1.$ The Clifford action of $\beta_3$ on $\phi$ is given by the following formula.  
\begin{align*}
    c(\beta_3)\phi=-\sqrt{2}i\pi_1^*(\bar\partial f_1\wedge\varphi)\otimes \pi_2^*e_1
\end{align*}
The formula follows from the formula of the Clifford action of an one-form \ref{Clifford one form} and the Clifford action of $\pi_2^*(dx_1\wedge dx_2)$ \ref{Clifford formula 5d 1}. We also compute the Dirac operator acting on $\phi.$ 
\begin{align*}
    D_A\phi=\pi_1^*\big(\sqrt{2}\bar\partial_{A_h} (e^{i(1-i)f_1}\varphi)\big)\otimes \pi_2^*e_1=\pi_1^*\big(\sqrt{2}i(1-i)e^{i(1-i)f_1}(\bar\partial f_1\wedge\varphi)\big)\otimes \pi_2^*e_1
\end{align*}
Hence $\phi,A=\pi_1^*(-A_{K_\Sigma}+2A_h),\beta_3=\pi_1^*(\bar\partial f_1+\partial\bar f_1)\wedge \pi_2^*(dx_1\wedge dx_2)$ and $\beta_5=0$ solves the Dirac equation \eqref{Dirac 5D perturbed}. To solve the curvature equation we compute
\begin{align*}
    d\beta_3=\pi_1^*\big(-\Delta(\text{Im}f_1)\omega\big)\wedge\pi_2^*(d x_1\wedge d x_2)\\
    d^*\beta_3=\pi_1^*\big(\Delta(\text{Re}f_1)\big)\pi_2^*(d x_1\wedge d x_2)
\end{align*}
Using \ref{Clifford formula 5d 1},\ref{Clifford formula 5d 2} we identify the quadratic term $q(\phi)$ to be
\begin{align*}
    q(\phi)=\frac{|\phi|^2}{4}\big(i\pi_1^*\omega+i\pi_2^*(d x_1\wedge dx_2)-\pi_1^*\omega\wedge \pi_2^*(d x_1\wedge dx_2)\big)
\end{align*}
the curvature equation \eqref{Curvature 5D perturbed} reads
\begin{align*}
    F_A-2i \pi_1^*\big(\Delta(\text{Re}f_1)\big)\pi_2^*(d x_1\wedge d x_2)-2\pi_1^*\big(\Delta(\text{Im}f_1)\omega\big)\wedge\pi_2^*(d x_1\wedge d x_2)\\+a(id x_1\wedge d x_2-\omega\wedge d x_1\wedge d x_2)\\
    =\frac{|\phi|^2}{4}\big(i\pi_1^*\omega+i\pi_2^*(d x_1\wedge dx_2)-\pi_1^*\omega\wedge \pi_2^*(d x_1\wedge dx_2)\big)
\end{align*}
This splits into three parts:
\begin{align}
    2F_{A_h}-F_{K_\Sigma}=\frac{i}{4}e^{2(\text{Re}f_1-\text{Im}f_1)}|\varphi|^2\omega,\label{curvature 5d 11}
    \\
    -2\Delta(\text{Im}f_1)+\frac{1}{4}e^{2(\text{Re}f_1-\text{Im}f_1)}|\varphi|^2=a,\label{curvature 5d 12}\\
    2\Delta(\text{Re}f_1)+\frac{1}{4}e^{2(\text{Re}f_1-\text{Im}f_1)}|\varphi|^2=a\label{curvature 5d 13}
\end{align}
Equations \eqref{curvature 5d 12} and \eqref{curvature 5d 13} give us
\begin{align*}
    \Delta (\text{Re}f_1+\text{Im}f_1)=0.\text{ So, }\text{Re}f_1+\text{Im}f_1=a_1,\text{ for a constant $a_1$}
\end{align*}
Hence, equations \eqref{curvature 5d 12} and \eqref{curvature 5d 13} become one single equation:
\begin{align}
    2\Delta(\text{Re}f_1)+\frac{1}{4}e^{4\text{Re}f_1}e^{-2a_1}|\varphi|^2=a\label{curvature 5d 14}
\end{align}
We perturb the initial metric $h$ on $\mathcal{L}$ by $e^\lambda,$ $\lambda\in C^\infty(\Sigma,\mathbb{R}).$ The new metric being $h':=e^\lambda h.$ This conformal change in the metric doesn't change the holomorphic structure of $L.$ We write the equations \eqref{curvature 5d 11},\eqref{curvature 5d 14} with respect to the Chern connection $A_{h'}$ and the metric $h'.$
\begin{align}
    2F_{A_{h'}}-F_{K_\Sigma}=\frac{i}{4}e^{4\text{Re}f_1}e^{-2a_1}|\varphi|^2_{h'}\omega\label{curvature 5d 15}\\
    2\Delta(\text{Re}f_1)+\frac{1}{4}e^{4\text{Re}f_1}e^{-2a_1}|\varphi|^2_{h'}=a\label{curvature 5d 16}
\end{align}
Since $F_{A_{h'}}=F_{A_h}-\partial\bar\partial \lambda,$ we get
\begin{align*}
     2F_{A_{h}}-2\partial\bar\partial\lambda-F_{K_\Sigma}=\frac{i}{4}e^{(4\text{Re}f_1-2a_1+\lambda)}|\varphi|^2\omega
\end{align*}
Taking point-wise inner product with $\omega$ on both sides we have
\begin{align}\label{eq 100}
    \langle (2F_{A_{h}}-F_{K_\Sigma}),\omega\rangle-2\langle\partial\bar\partial\lambda,\omega\rangle=\frac{i}{4}e^{(4\text{Re}f_1-2a_1+\lambda)}|\varphi|^2
\end{align}
Use lemma \ref{lemma 2} to write $2\langle\partial\bar\partial\lambda,\omega\rangle=i\Delta\lambda$. Hence equation \eqref{eq 100} can be written as
\begin{align}
     \Delta\lambda+\frac{e^{-2a_1}|\varphi|^2}{4}e^{(4\text{Re}f_1+\lambda)}=i\langle F_{K_\Sigma},\omega\rangle-2i\langle F_{A_{h}},\omega\rangle\label{curvature 5d 17}
\end{align}
Multiplying both sides of equation \eqref{curvature 5d 16} by $2$ and adding it to equation \eqref{curvature 5d 17} we get
\begin{align}
    \Delta(4\text{Re}f+\lambda)+\frac{3e^{-2a_1}|\varphi|^2}{4}e^{(4\text{Re}f_1+\lambda)}=2a+i\langle F_{K_\Sigma},\omega\rangle-2i\langle F_{A_{h}},\omega\rangle\label{curvature 5d 18}
\end{align}
\begin{align*}
    a=\frac{2\pi}{\text{vol}(\Sigma)}\text{deg}(K_\Sigma-2\mathcal{L})>0\text{ implies } \int_\Sigma \big(2a+i\langle F_{K_\Sigma},\omega\rangle-2i\langle F_{A_h},\omega\rangle\big)\omega=0
\end{align*} 
Since $|\varphi|^2$ has isolated zeroes, $(4$Re$f+\lambda)$ has a unique solution from equation \eqref{curvature 5d 18} \cite{KW}. On the other hand, Multiplying both sides of equation \eqref{curvature 5d 17} by $-1$ and adding it to equation \eqref{curvature 5d 16} we get 
\begin{align}
    \Delta(2\text{Re}f-\lambda)=a-i\langle F_{K_\Sigma},\omega\rangle+2i\langle F_{A_h},\omega\rangle\label{eq 1000}
\end{align}
\begin{align*}
    a=\frac{2\pi}{\text{vol}(\Sigma)}\text{deg}(K_\Sigma-2\mathcal{L})\text{ implies } \int_\Sigma \big(a-i\langle F_{K_\Sigma},\omega\rangle+2i\langle F_{A_h},\omega\rangle\big)\omega=0
\end{align*}
Therefore, equation \eqref{eq 1000} has a unique solution for $(2$Re$f-\lambda)$. Since we have solved both the functions \((4\text{Re}f + \lambda)\) and \((2\text{Re}f - \lambda)\), we can determine \(\text{Re}f\) and \(\lambda\) individually. Finally, we perform a \(\mathbb{C}^*\) gauge transformation to obtain a solution with respect to the original metric \(h\).
\par
\ul{$\eta= a(id x_1\wedge d x_2-\omega\wedge d x_1\wedge d x_2),a=-\frac{2\pi}{\text{vol}(\Sigma)}$deg$(K_\Sigma-2\mathcal{L})<0:$} This case is almost identical to the previous one. We take a non-trivial holomorphic section $\varphi$ of $\mathcal{L}$ and an $f_2\in C^\infty(\Sigma,\mathbb{C}).$ Then for $\phi:=\pi_1^*(e^{i(1-i)f_2}\varphi)\otimes \pi_2^*e_2,$ $\beta_3=-(\bar\partial f_2+\partial\bar {f_2})\wedge d x_1\wedge dx_2$ and $\beta_5=0;$ $\phi,A=\pi_1^*(-A_{K_\Sigma}+2A_h),\beta_3,\beta_5$ solves the Dirac equation \eqref{Dirac 5D perturbed}.\par 
Using \ref{Clifford formula 5d 1} and \ref{Clifford formula 5d 2} we deduce
\begin{align*}
    q(\phi)=\frac{|\phi|^2}{4}\big(i\pi_1^*\omega-i\pi_2^*(dx_1\wedge dx_2)+\pi_1^*\omega\wedge \pi_2^*(dx_1\wedge dx_2) \big)
\end{align*}
Hence the curvature equation \eqref{Curvature 5D perturbed} breaks down into three parts:
\begin{align}
    2F_{A_h}-F_{K_\Sigma}=\frac{i}{4}e^{2(\text{Re}f_2-\text{Im}f_2)}|\varphi|^2\omega\label{eq 1004}\\
    -2\Delta(\text{Im}f_2)+\frac{1}{4}e^{2(\text{Re}f_2-\text{Im}f_2)}|\varphi|^2=-a\\
    2\Delta(\text{Re}f_2)+\frac{1}{4}e^{2(\text{Re}f_2-\text{Im}f_2)}|\varphi|^2=-a
\end{align}
The last two equations give us 
\begin{align*}
    \Delta(\text{Re}f_2+\text{Im}f_2)=0. \text{ So, } \text{Re}f_2+\text{Im}f_2=a_2,\text{  for a constant $a_2$}
\end{align*}
and hence become one single equation:
\begin{align}
    2\Delta(\text{Re}f_2)+\frac{1}{4}e^{4\text{Re}f_2}e^{-2a_2}|\varphi|^2=-a\label{eq 1005}
\end{align}
One can solve the two equations \eqref{eq 1004},\eqref{eq 1005} exactly in the same way as done in the previous case. The integrability criteria being $a=-\frac{2\pi}{\text{vol}(\Sigma)}$deg$(K_\Sigma-2\mathcal{L})<0.$\par
\ul{$\eta= a(id x_1\wedge d x_2+\omega\wedge d x_1\wedge d x_2),a=-\frac{2\pi}{\text{vol}(\Sigma)}$deg$(K_\Sigma-2\mathcal{L})>0:$} We start with a non-trivial section $\psi\in\Omega^{0,1}(X,\mathcal{L})$ such that $\bar\partial_{A_h}^*\psi=0$. In another words $\psi$ is an anti-holomorphic section of $K_\Sigma^{-1}\otimes \mathcal{L}=\mathcal{L}-K_\Sigma.$ Hence $\bar\psi$ is a holomorphic section of $K_\Sigma-\mathcal{L}.$ For an $f_3\in C^\infty(\Sigma,\mathbb{C})$, we define a spinor $\phi:=\pi_1^*(e^{-(1+i)\bar f_3}\psi)\otimes \pi_2^*e_1,\beta_3=-\pi_1^*(\bar\partial f_3+\partial\bar f_3)\wedge \pi_2^*(dx_1\wedge dx_2)$ and $\beta_5=0.$ Define $A:=\pi_1^*(-A_{K_\Sigma}+2A_h).$ We compute $D_A(\phi):$
\begin{align*}
    D_A\big(\pi_1^*(e^{-(1+i)\bar f_3}\psi)\otimes \pi_2^*e_1\big)&=\sqrt{2}\pi_1^*\big(\bar\partial_{A_h}^*(e^{-(1+i)\bar f_3}\psi)\big)\otimes \pi_2^*e_1\\
    &=-\sqrt{2}(1-i)\pi_1^*\big(e^{-(1+i)\bar f_3}*(\partial \bar f_3\wedge\psi)\big)\otimes\pi_2^* e_1
\end{align*}
The Clifford action of $\beta_3$ on $\phi$ is given by
\begin{align*}
    c\big(-\pi_1^*(\bar\partial f_3+\partial\bar f_3)\wedge \pi_2^*(dx_1\wedge dx_2)\big)\big(\pi_1^*(e^{-(1+i)\bar f_3}\psi)\otimes \pi_2^*e_1\big)\\
    =\sqrt{2}\pi_1^*\big(e^{-(1+i)\bar f_3}*(\partial\bar f_3\wedge \psi)\big)\otimes \pi_2^*e_1
\end{align*}
The formula follows from the formula of the Clifford action of an one-form \ref{Clifford one form} and the Clifford action of $\pi_2^*(dx_1\wedge dx_2)$ \ref{Clifford formula 5d 1}. Hence, $\phi=\pi_1^*(e^{-(1+i)\bar f_3}\psi)\otimes \pi_2^*e_1,A=\pi_1^*(-A_{K_\Sigma}+2A_h),\beta_3=-\pi_1^*(\bar\partial f_3+\partial\bar f_3)\wedge \pi_2^*(dx_1\wedge dx_2)$ and $\beta_5=0$ solves the Dirac equation \eqref{Dirac 5D perturbed}.\par
Since $\phi\in\Gamma\big(\pi_1^*\Lambda^{0,1}(\Sigma,\mathcal{L})\otimes\pi_2^*\langle e_1\rangle\big)$, using \ref{Clifford formula 5d 1},\ref{Clifford formula 5d 2} we observe:
\begin{align*}
    q(\phi)&=\frac{|\phi|^2}{4}\big(-i\pi_1^*\omega+i\pi_2^*(dx_1\wedge dx_2)+\pi_1^*\omega\wedge \pi_2^*(dx_1\wedge dx_2)\big)
\end{align*}
Thereafter, the curvature equation \eqref{Curvature 5D perturbed} breaks into three parts:
\begin{align}
    2F_{A_h}-F_{A_{K_\Sigma}}=-\frac{i}{4}e^{-2(\text{Re}f_3+\text{Im}f_3)}|\psi|^2\omega\label{eq 90}\\
    -2\Delta(\text{Im}f_3)+\frac{1}{4}e^{-2(\text{Re}f_3+\text{Im}f_3)}|\psi|^2=a\\
    -2\Delta(\text{Re}f_3)+\frac{1}{4}e^{-2(\text{Re}f_3+\text{Im}f_3)}|\psi|^2=a
\end{align}
The last two equations give us $\Delta\big(\text{Re}f_3-\text{Im}f_3\big)=0.$ Hence, $\text{Re}f_3=\text{Im}f_3- a_3,$ for a constant $a_3$. So, they become one single equation:
\begin{align}\label{eq 101}
    -2\Delta(\text{Re}f_3)+\frac{1}{4}e^{-4(\text{Re}f_3)}e^{-2a_3}|\psi|^2=a
\end{align}
We perturb the initial metric $h$ on $\mathcal{L}$ by $e^{\tilde\lambda},$ $\tilde\lambda\in C^\infty(\Sigma,\mathbb{R}).$ The new metric being $\tilde h:=e^{\tilde\lambda} h.$ Notice this conformal change in the metric doesn't change the holomorphic structure of $\mathcal{L}.$ This conformal change in metric of $\mathcal{L},$ induces a conformal change by $e^{-\tilde\lambda}$ on the corresponding hermitian metric on $\mathcal{L}^*\cong \mathcal{L}^{-1}.$ So, $\bar\psi$ remains a holomorphic section of $K_\Sigma-\mathcal{L},$ and now when we go back to the corresponding anti-holomorphic section on $\mathcal{L}-K_\Sigma,$ the norm changes by $e^{-\tilde\lambda}.$ We write equations \eqref{eq 90},\eqref{eq 101} with respect to the Chern connection $A_{\tilde h}$ and the metric $\tilde h$ on $\mathcal{L}:$ 
\begin{align}
    2F_{A_{\tilde h}}-F_{K_\Sigma}=-\frac{i}{4}e^{-4\text{Re}f_3}e^{-2a_3}e^{-\tilde\lambda}|\psi|^2\omega\label{curvature 5d case 3}\\
    -2\Delta(\text{Re}f_3)+\frac{1}{4}e^{-4(\text{Re}f_3)}e^{-2a_3}e^{-\tilde\lambda}|\psi|^2=a\label{curvature 5d case 31}
\end{align}
Since $F_{A_{\tilde{h}}}=F_A-\partial\bar\partial\tilde{\lambda},$ equation \eqref{curvature 5d case 3}
reads
\begin{align*}
    2F_A-2\partial\bar\partial \tilde\lambda-F_{K_\Sigma}=-\frac{i}{4}e^{-4\text{Re}f_3}e^{-2a_3}e^{-2\tilde\lambda}|\psi|^2\omega
\end{align*}
Finally taking point wise inner product with the volume form $\omega$ and rearranging some terms we get
\begin{align}
\Delta (-\tilde\lambda)+\frac{1}{4}e^{-4\text{Re}f_3}e^{-2c_3}e^{-\tilde\lambda}|\psi|^2=i\langle (2F_{A_{h}}-F_{K_\Sigma}),\omega\rangle\label{eq 110}
\end{align}
Multiplying both sides of equation \eqref{curvature 5d case 31} by $2$ and adding it to equation \eqref{eq 110} we get
\begin{align}
\Delta(-4\text{Re}f_3-\tilde{\lambda})+\frac{3e^{-2a_3}|\psi|^2}{4}e^{(-4\text{Re}f_3-\tilde{\lambda})}=2a+i\langle (2F_{A_{h}}-F_{K_\Sigma}),\omega\rangle\label{eq 79}
\end{align}
Notice that
\begin{align*}
    a=-\frac{2\pi}{\text{vol}(\Sigma)}\text{deg}(K_\Sigma-2\mathcal{L})>0\text{ implies }\int_\Sigma \big(2a+i\langle F_{K_\Sigma},\omega\rangle-2i\langle F_{A_h},\omega\rangle\big)\omega>0.
\end{align*}
Since $|\psi|^2$ has isolated zeroes, we have a unique solution for $(-4\text{Re}f_3-\tilde{\lambda})$ from equation \eqref{eq 79} \cite{KW}. 
On the other hand, Multiplying both sides of equation \eqref{curvature 5d case 31} by $-1$ and adding it to equation \eqref{eq 110} we get
\begin{align}
    \Delta(2\text{Re}f_3-\tilde\lambda)=-a+i\langle (2F_{A_{h}}-F_{K_\Sigma}),\omega\rangle\label{eq 80}
\end{align}
\begin{align*}
    a=-\frac{2\pi}{\text{vol}(\Sigma)}\text{deg}(K_\Sigma-2\mathcal{L})\text{ implies }\int_\Sigma \big(-a+i\langle F_{K_\Sigma},\omega\rangle-2i\langle F_{A_h},\omega\rangle\big)\omega=0.
\end{align*}
Hence we have a unique solution for $(2\text{Re}f_3-\tilde{\lambda})$ from equation \eqref{eq 80}. Since we have solved both the functions \((-4\text{Re}f_3 -\tilde\lambda)\) and \((2\text{Re}f_3 - \tilde\lambda)\), we can determine \(\text{Re}f_3\) and \(\tilde\lambda\) individually. Finally, we perform a \(\mathbb{C}^*\) gauge transformation to obtain a solution with respect to the original metric \(h\).\par
\ul{$\eta= a(id x_1\wedge d x_2+\omega\wedge d x_1\wedge d x_2),a=\frac{2\pi}{\text{vol}(\Sigma)}\text{deg}(K_\Sigma-2\mathcal{L})<0:$} For a function $g_2\in C^\infty(\Sigma,\mathbb{C}),$ define $\phi:=\pi_1^*\big(e^{-(1+i)\bar g_2}\psi\big)\otimes \pi_1^*e_2,\beta_3:=\pi_1^*(\bar\partial g_2+\partial\bar {g_2})\wedge \pi_2^*(dx_1\wedge dx_2)$ and $\beta_5:=0.$ Observe that $\phi,A=\pi_1^*(-A_{K_\Sigma}+2A_h),\beta_3,\beta_5$ solves the Dirac equation \eqref{Dirac 5D perturbed}.\par
Using \ref{Clifford formula 5d 1},\ref{Clifford formula 5d 2} we identify the quadratic term $q(\phi):$
\begin{align*}
    q(\phi)=-\frac{1}{4}|\phi|^2(i\omega+id x_1\wedge d x_2+\omega\wedge d x_1\wedge dx_2)
\end{align*}
Thereafter the curvature equation \eqref{Curvature 5D perturbed} breaks down into three parts:
\begin{align}
    2F_{A_h}-F_{K_\Sigma}=-\frac{i}{4}e^{-2(\text{Re}g_2+\text{Im}g_2)}|\psi|^2\omega\label{eq 81}\\
    -2\Delta(\text{Im}g_2)+\frac{1}{4}e^{-2(\text{Re}g_2+\text{Im}g_2)}|\psi|^2=-a\\
    -2\Delta(\text{Re}g_2)+\frac{1}{4}e^{-2(\text{Re}g_2+\text{Im}g_2)}|\psi|^2=-a
\end{align}
From the last two equations we get $\Delta(\text{Re}g_2-\text{Im}g_2)=0$. Hence,  $\text{Re}g_2=\text{Im}g_2+a_4,$\text{ for a constant $a_4$}. Thereafter they turn into one single equation:
\begin{align}
    -2\Delta(\text{Re}g_2)+\frac{1}{4}e^{-4\text{Re}g_2}e^{-2a_4}|\psi|^2=-a\label{eq 82}
\end{align}
One can solve the two equations \eqref{eq 81},\eqref{eq 82} exactly in the same way as done in the previous case. The integrability criteria being $a=\frac{2\pi}{\text{vol}(\Sigma)}\text{deg}(K_\Sigma-2\mathcal{L})<0$.
\end{proof}
\subsection{Relationship with vortices}
We start with the first case, i.e., $\eta= a(id x_1\wedge d x_2-\omega\wedge d x_1\wedge d x_2),a=\frac{2\pi}{\text{vol}(\Sigma)}\text{deg}(K_\Sigma-2\mathcal{L})>0.$ From the proof above notice that there exists a hermitian metric on $\mathcal{L}$ such that for a non-trivial holomorphic section $\varphi\in H^0(\Sigma,\mathcal{L})$, there exists an $f_1\in C^\infty(\Sigma,\mathbb{C})$ and a unitary connection $A_0$ on $\mathcal{L},
\big(\phi=\pi_1^*(e^{i(1-i)f_1}\varphi\otimes\pi_2^* e_1,A=\pi_1^*(-A_\Sigma+2A_0),\beta_3=\pi_1^*(\bar\partial f_1+\partial\bar f_1)\wedge \pi_2^*(dx_1\wedge dx_2),\beta_5=0\big)$ solves the Seiberg--Witten equations \eqref{Dirac 5D perturbed},\eqref{Curvature 5D perturbed}.\par
For $c\in\mathbb{C}^*,$ define 
\begin{align*}
    f_c=f_1+\frac{(i-1)}{2}\text{ln}|c|
\end{align*}
Following the proof we observe that instead of $\varphi,$ if we started with $c\varphi\in H^0(\Sigma,\mathcal{L}),$ we would get another solution of the Seiberg--Witten equation, namely, $\big(\phi=\pi_1^*(e^{i(1-i)f_c}c\varphi\otimes\pi_2^* e_1,A=\pi_1^*(-A_\Sigma+2A_0),\beta_3=\pi_1^*(\bar\partial f_c+\partial\bar f_c)\wedge \pi_2^*(dx_1\wedge dx_2),\beta_5=0\big).$ But we have 
\begin{align*}
    e^{i(1-i)f_c}c\varphi=\frac{c}{|c|}e^{i(1-i)f_1}\varphi\\
    (\bar\partial f_c+\partial\bar f_c)=(\bar\partial f_1+\partial\bar f_1)
\end{align*}
Hence the new solution is gauge-equivalent to the one before. Thereafter the construction of the solution in the proof of theorem \ref{Theorem 5d} relies solely on the holomorphic structure of $\mathcal{L}$ and the conformal class of a holomorphic section of $\mathcal{L}.$ Therefore, we can identify the moduli space of constructed solutions with $\mathcal{M}_{\text{vortex}}\big(c_1(\mathcal{L})\big).$\par
Similar observations can be made for the solutions in the other cases. For $\eta= a(id x_1\wedge d x_2-\omega\wedge d x_1\wedge d x_2),a=-\frac{2\pi}{\text{vol}(\Sigma)}$deg$(K_\Sigma-2\mathcal{L})<0,$ the moduli space of constructed solutions can be identified with $\mathcal{M}_{\text{vortex}}(\mathcal{L}).$ For the rest of the two cases where deg$(K_\Sigma-2\mathcal{L})<0$, the moduli space of constructed solutions can be identified with $\mathcal{M}_{\text{vortex}}(K_\Sigma-\mathcal{L}).$
\printbibliography[
heading=bibintoc,
title={Bibliography}
]
\Addresses
\end{document}